\documentclass[12pt]{article} % A4 paper and 11pt font size

\usepackage[english]{babel}
\usepackage[utf8]{inputenc}
\usepackage[T1]{fontenc}
\usepackage{amsmath,amsfonts,amsthm,amssymb}
\usepackage{stmaryrd }
\usepackage{bbm}
\usepackage{mathtools}

\DeclareUnicodeCharacter{FFFD}{ }

\usepackage{wrapfig}
\usepackage{caption}
\usepackage{subcaption}
\usepackage{graphicx}
\usepackage{float}
\usepackage{adjustbox}

\usepackage{blindtext} %for enumarations
\usepackage{enumitem}

\usepackage[]{hyperref}  %link collor
\usepackage[margin=2.5cm]{geometry}

\usepackage{xcolor}
\usepackage{soul}

\usepackage{tkz-berge}
\usepackage{tikz}
\usepackage{tikz-cd}
\usetikzlibrary{arrows,automata, graphs}
\usetikzlibrary{positioning}
\usetikzlibrary{patterns}
\usetikzlibrary{decorations.pathmorphing}
\usepackage{pgf}
\tikzcdset{
  cells={font=\everymath\expandafter{\the\everymath\displaystyle}},
}
\tikzset{
    labl/.style={near start, rotate=-45}
}
\tikzset{
    state/.style={
           rectangle,
           inner sep=2pt,
           text centered,
           },
}

\usepackage[french]{algorithm2e}

\usepackage{multicol}

\usepackage[
backend=biber,
style=numeric,
sorting=nyt,
maxbibnames = 99,
]{biblatex}
\usepackage{csquotes}
\usepackage{fancyhdr} % Custom headers and footers
\newcommand{\horrule}[1]{\rule{\linewidth}{#1}} % Create horizontal rule command with 1 argument of height

\title{Fractionally Calabi--Yau lattices that tilt to higher Auslander algebras of type A}
\author{Tal GOTTESMAN} % Your name
\date{\normalsize\today} % Today's date or a custom date

\theoremstyle{plain}
\newtheorem{prop}{Proposition}[subsection]
\newtheorem{lem}[prop]{Lemma}

\newtheorem{thm}[prop]{Theorem}
\newtheorem*{thm*}{Theorem}
\newtheorem*{prop*}{Proposition}
\newtheorem{cor}[prop]{Corollary}
\newtheorem{thmIntro}{Theorem}

\newtheorem{corIntro}[thmIntro]{Corollary}

\newtheorem{quest}{Question}

\theoremstyle{definition}
\newtheorem{Def}[prop]{Definition}
\newtheorem{ex}[prop]{Example}

\newtheorem{rem}[prop]{Remark}
\newtheorem{nota}[prop]{Notation}

 \DeclareMathOperator{\Ima}{Im}
\DeclareMathOperator{\HH}{H}

\DeclareMathOperator{\Hom}{Hom}
\DeclareMathOperator{\End}{End}

\DeclareMathOperator{\HomotopyCat}{Ho}
\DeclareMathOperator{\Thicc}{Thick}
\DeclareMathOperator{\Serre}{\mathbb{S}}

\newcommand{\PRes}{\mathcal{P}}

\newcommand{\field}{\mathbbm{k}}
\newcommand{\IncAlg}{\mathcal{A}}

\definecolor{darkblue}{rgb}{0.0,0,0.7} % darkblue color
\definecolor{darkred}{rgb}{0.7,0,0} % darkred colorexplici
\newcommand{\darkred}{\color{darkred}} % darkred command
\newcommand{\defn}[1]{\emph{\darkred #1}} % emphasis of a

\newif\ifBelowAlpha
\BelowAlphafalse

\begin{document}
%\nocite{*}

\maketitle % Print the title
\begin{abstract}
We prove that the bounded derived category of the lattice of order ideals of the product of two ordered chains is fractionally Calabi--Yau. We also show that these lattices are derived equivalent to higher Auslander algebras of type A. The proofs involve the study of intervals of the poset that have resolutions described with antichains having rigid properties. These two results combined corroborate a conjecture by Chapoton linking posets to Fukaya--Seidel Categories.
\end{abstract}

\tableofcontents
%%%%%%%%%%%%%%%%%%%%%%%%%%%%%%%%%%%%%%%%%%%%%%%%%%%%%%%%%%%%%%%

\label{sec:notations}
\section{Introduction}
The notion of fractionally Calabi--Yau categories was introduced by Kontsevich in the late nineties \cite[Definition ~28]{Kontsevich98}. A triangulated category $\mathcal{T}$ with a Serre functor $\Serre$ is said to be \defn{fractionally Calabi--Yau} if there exists $l$ and $d$ such that $\Serre^{l}$ is isomorphic as a functor to the suspension functor applied $d$ times. We say that $\mathcal{T}$ is $\frac{d}{l}$-Calabi--Yau.
When $\mathcal{T} = D^b(A)$, the bounded derived category of an algebra $A$ of finite global dimension over a field $\field$, we can take $\Serre = -\otimes_A^{\mathbbm{L}} DA$, the derived Nakayama functor. In that case, the Calabi--Yau property can be further relaxed. If $\phi$ is an automorphism of $A$, then $D^b(A)$ is said to be twisted fractionally Calabi--Yau if $\Serre^{l}\simeq [d]\circ \phi^{*}$, where $\phi^*$ twists the action of $A$ on a module by $\phi$. We recover the previous definition when $\phi = id_A$. The following theorem makes it easier to detect twisted fractionally Calabi--Yau algebras.

\begin{thmIntro}[{\cite[Proposition ~4.3]{Herschend_2011}}]
Let $\Lambda$ be a finite dimensional $\field$-algebra of finite global dimension. The following conditions are equivalent.
\begin{enumerate}[label=(\roman*)]
\item $\Lambda$ is twisted $\frac{d}{l}$-Calabi--Yau.
\item $\Serre^{l}\Lambda\simeq \Lambda [d]$.
\end{enumerate}
\end{thmIntro}
This leads to a question which is still far from being answered in general.
\begin{quest}[{\cite[Remark ~1.6]{Herschend_2011}}]\label{quest:tfCY}Is every twisted fractionally Calabi--Yau algebra fractionally Calabi--Yau?
\end{quest}
Further, the trivial extension algebra of a (twisted) fractionally Calabi--Yau finite dimensional algebras of finite global dimension is (twisted) periodic \cite{chan2024periodic}. Hence Question \ref{quest:tfCY} is linked to the following conjecture of Erdmann and Skowro{\'n}ski \cite{Erdmann2015}. 
\begin{quest}[{\cite[Question ~1.4]{chan2024periodic}}]\label{quest:tPer}Is every finite-dimensional twisted periodic algebra periodic?
\end{quest}
In the case of finite posets with a unique maximal element or a unique minimal element, the answer to Question \ref{quest:tfCY} is positive as per \cite[Theorem~3.1]{rognerud2018bounded}. However, for a given incidence algebra the existence of an isomorphism \[\Serre^l (A)\simeq A[d]\] is still in general very hard to check \cite{rognerud2018bounded, Yildirim_2018}. In this article we provide a relaxation of \cite[Theorem~3.1]{rognerud2018bounded} to help overcome that difficulty.
\begin{thmIntro}Let $L$ be a finite lattice, $d$ and $l$ integers and $(C_{\alpha})_{\alpha\in L}$ be a family of indecomposable modules with simple head $S_{\alpha}$ and a boolean resolution. If for all $\alpha\in L$ it holds that $\Serre^{l}(C_{\alpha})\simeq C_{\alpha}[d]$, then $L$ is $\frac{d}{l}$- fractionally Calabi--Yau.
\end{thmIntro}
This theorem does not provide with a recipe to find such appropriate families, but it suggests certain criteria which restrict the search for the perfect combinatorial candidates.

Fractionally Calabi--Yau posets are fascinating objects in part because of a hypothetical relation to \defn{product formulas} due to Chapoton \cite{chapoton2023posets}. In combinatorics, many families of sets $(S_n)_{n\in \mathbb{N}}$ can be counted by product formulas $|S_n| = \Pi_{i=1}^{n}\frac{D-d_i}{d_i}$ where the sum of the numerator and denominator is constant and equal to $D$. Such families include the Catalan numbers, the number of alternating sign matrices, the West family and the Tamari intervals family. Chapoton's highly conjectural explanation is that there should exist a partial order on $S_n$ whose derived category is $\frac{C}{D}$-Calabi--Yau, where $C = \sum_i D-2d_i$. Moreover, the bounded derived category in question should be equivalent to a type of Fukaya--Seidel category constructed from the data of $D$ and the $d_i$ coefficients. The conjecture also provides predictions regarding the Coxeter polynomial and the Milnor number of the singularity some of which can be tested with a computer on examples. Some consequences of these conjectures have been proven since (\cite{rognerud2018bounded}). The starting point of this paper was to prove another one of these resulting conjectures which was already studied in part in \cite{Yildirim_2018}. 
Observe that the binomial ${m+n\choose m}$ can be written as
\begin{equation}
\frac{m+n}{1}\frac{m+n-1}{2}\cdot \dots\cdot\frac{m + 1}{n}
\end{equation}
where $D = m + n + 1$. This is probably one of the most natural examples of product formula discussed above. The poset of order ideals of a product of total orders of length $m$ and $n$ has cardinality ${n+m \choose m}$ and we write it $J_{m,n}$. Using our results on \defn{boolean antichain modules} we are able to confirm Chapoton's prediction about the Calabi--Yau dimension of these posets. 
\begin{thmIntro}\label{thm1}The bounded derived category of $J_{m,n}$ is $\frac{mn}{m+n+1}$-Calabi--Yau.
\end{thmIntro}
As a corollary this gives a positive answer to the Chapoton-Y\i ld\i r\i m conjecture on cominuscule posets of type A and B \cite{Yildirim_2018}. 
\begin{corIntro}\label{thm:corConjYildChap}The bounded derived category of cominuscule posets of type $A$, $B$, $D$ are fractionally Calabi--Yau. For types A and B, the denominator is $h + 1$ where h is a constant associated with the root system. %The numerator depends on the root for type A. For type B it is $h+1$ as well.
\end{corIntro}
The key observation one needs for applying Theorem \ref{thm1} to cominuscule posets  is their classification into types $C_{I}$, $C_{II}$ or $C_{III}$ depicted below (\cite{thomas2009cominuscule}).
\begin{figure}[h]
	\begin{minipage}{0.3\textwidth}\centering
		\begin{tikzpicture}
		\draw[step=0.5cm, rotate = 45] (0, 0) grid (1.5, 1.5);
		\draw[step=0.5cm, rotate = 45] (1.99, 0) grid (3, 1.5);
		\draw[step=0.5cm, rotate = 45] (0, 1.99) grid (1.5, 3);
		\draw[step=0.5cm, rotate = 45] (1.99, 1.99) grid (3, 3);
		\draw[step=0.5cm, rotate = 45, dotted] (0, 0) grid (3, 3);
		\end{tikzpicture}
		\caption*{$C_{I}$}\label{fig:c1}
	\end{minipage}\hfill
	\begin{minipage}{0.3\textwidth}\centering
		\begin{tikzpicture}
		\draw[step=0.5cm, rotate = -45] (-3, 0) grid (-2, 1);
		\draw[step=0.5cm, rotate = -45, dotted] (-3, 0) grid (-1.5, 1.5);
		\draw[step=0.5cm, rotate = -45] (-1.5, 0) grid (-0.5, 0.5);
		\draw[step=0.5cm, rotate = -45] (-1.5, 0.5) grid (-1, 1);
		\draw[rotate = -45] (0,0) -- (-0.5, 0);
		\draw[rotate = -45, step = 0.5cm] (-3, 1.49) grid (-2, 2);
		\draw[step = 0.5cm, rotate = -45] (-3, 2) grid (-2.5, 2.5);
		\draw[step = 0.5cm, rotate = -45](-3, 2.5) -- (-3, 3);
		\end{tikzpicture}
		\caption*{$C_{II}$}\label{fig:c2}
	\end{minipage}\hfill
	\begin{minipage}{0.3\textwidth}\centering
		\begin{tikzpicture}
		\draw[rotate = 45, step = 0.5] (0,0) -- (0.5,0.5);
		\draw[rotate = 45, step = 0.5, dotted] (0.5,0.5) -- (1, 1);
		\draw[rotate = 45, step = 0.5, dotted] (1.5,1.5) -- (2,2);
		\draw[rotate = 45, step = 0.5] (2,2) -- (2.5, 2.5);
		\draw[rotate = 45, step = 0.5] (0.99,0.99) grid (1.51, 1.51);
		\end{tikzpicture}
		\caption*{$C_{III}$}\label{fig:c3}
	\end{minipage}
	\caption{The three types of cominuscule posets}\label{fig:tripplethreat}	
\end{figure}

Interestingly, there is no one to one correspondance between this classification and the ADE classification of the root poset one started with. However cominuscule posets of type $A$ and $B$ all follow the pattern of  $C_{I}$. Corollary \ref{thm:corConjYildChap} follows from that. Type $D$ follows pattern $C_{III}$ which is proved using Ladkani's flip flop techniques \cite{ladkani2007flipflop}. Type $C$ follows pattern $C_{II}$ and seems to be different and is not a consequence of our work. The conjecture is still open in this case.

Our proof of Theorem \ref{thm1} gives a good understanding of the Serre functor for this category. However, one would like to have a more structural reason behind the Calabi--Yau property for posets. For us, a good reason why $J_{m,n}$ should be fractionally Calabi--Yau is our second main result which is the following derived equivalence. 
\begin{thmIntro}\label{thm:DerEqTypeA} The algebra of the poset $J_{m,n}$ is derived equivalent to the higher Auslander algebra $A_{m+1}^{n-1}$. \end{thmIntro}
Higher Auslander algebras were introduced by Iyama in \cite{iyama2007auslander1} as part of a series of seminal articles on higher representation theory. Higher Auslander algebras of type A were soon after described in \cite{iyama2008typeA} and are known to be fractionally Calabi--Yau. As a corollary of Theorem \ref{thm1} and Theorem \ref{thm:DerEqTypeA} we have a new proof of an already known theorem.
\begin{thmIntro}\label{thm:fCYTypeA} Higher Auslander algebras of type $A$ are fractionally Calabi--Yau.
\end{thmIntro}
Previous proofs of this result have different flavours. The first one stemmed from symplectic geometry \cite{dyckerhoff_jasso_lekili_2021}, the second one, from the theory of infinity categories \cite{jasso2019CY1} and the most recent from an intricate study of the properties of a certain preprojective algebra and its Nakayama automorphism, linking it to the the Serre functor \cite{grant2020functors}. The proof presented here is more combinatorial. Of course knowing Theorems \ref{thm:DerEqTypeA} and \ref{thm:fCYTypeA} also gives a proof of Theorem \ref{thm1}.
 It is also satisfying to note that this derived equivalence ties back into the Chapoton conjectures as a partially wrapped Fukaya category can be associated to higher Auslander algebras of type $A$ \cite{dyckerhoff_jasso_lekili_2021}. A more recent preprint \cite{didedda2023symplectic} also links the higher Auslander algebras of type A to Fukaya Seidel Categories with the Milnor number predicted by Chapoton. 
\paragraph{Acknowledgements} I would like to thank my supervisor Baptiste Rognerud for introducing me to the subject as well as for all the discussion, guidance and careful reading of my work at every step of the way. I would like to thank my second supervisor, Bernhard Keller for Figure \ref{fig:keller}. I also want to thank the anonymous referees for their useful remarks and pointing out a missing references to \cite{dyckerhoff_jasso_lekili_2021} in the last section. This work constitutes part of my PhD thesis \cite{gottesman2024these} and is the object of a published extended abstract \cite{gottesman2024fpsac}.
\subsection{Notation}
\paragraph{Generalities} Let $\field$ be a field and $X$ a finite partially ordered set or poset. Define its incidence algebra $\IncAlg = \IncAlg_{\field}(X)$ over $\field$ to be the $\field$-vector space with basis pairs $(x, y)$ such that $x\leq y$ with multiplication defined by 
	\begin{equation*}
	(x, y)(z, t) = \begin{cases}(x, t)& \text{if } y=z,\\
				0 &\text{otherwise}.\end{cases}
	\end{equation*}
For $x\in X$ we write $e_x = (x, x)$ the primitive idempotent. Then we have $1_\IncAlg = \sum_{x\in X} e_x$.
Throughout this work we consider finite dimensional left modules over $\IncAlg$. For every element $x\in X$ the associated simple module is $S_x \cong \field$ with action $(y, t)\cdot 1_{\field} = 0$ unless $y=t=x$ in which case $e_x\cdot 1_{\field} = 1_{\field}$. Its projective cover $P_x = \mathcal{A}\cdot e_x$ has basis $\{(y, x)|y\leq x\}$.  Its injective hull is the injective indecomposable $I_x = (e_x \cdot\mathcal{A})^*$ and has basis $\{(x, y)^*| x\leq y\}$. Recall that morphisms between the projective indecomposables are characterised by
	\begin{equation*}
	\Hom_{\IncAlg} (P_x, P_y) = \Hom_{\IncAlg}(\mathcal{A} e_x, \mathcal{A} e_y)\cong\begin{cases}e_x \mathcal{A} e_y \cong \field&\text{if }x\leq y,\\0 &\text{otherwise.}\end{cases}
	\end{equation*}
We denote the canonical inclusion as $\iota_x^y: P_x\hookrightarrow P_y$ whenever $x\leq y$ which is the right multiplication by $(x, y)$. More generally for any left $\IncAlg$-module $M$, we have $\Hom_{\IncAlg}(P_x, M)\cong e_xM$. This isomorphism makes the following diagram commute
	\begin{equation}\label{eq:sqIota}
	\begin{tikzcd}
	f\ar[d, mapsto]\ar[r, phantom, "\in" description]&\Hom_{\IncAlg} (P_x, M)\ar[d]&\ar[l, "\circ \iota_x^y"]\ar[d] \Hom_{\IncAlg} (P_y, M)\ar[r, phantom, "\ni" description]&g\ar[d, mapsto]\\
f(e_x)\ar[r, phantom, "\in" description, near start]&e_xM&\ar[l, "{(x,y)\cdot}"]e_yM\ar[r, phantom, "\ni" description, near end]&g(e_y)
	\end{tikzcd}
	\end{equation}
The total hom complex $\Hom^{\bullet}_{\IncAlg}(C, M)$ where $C$ is a chain complex $C = ((C_n)_n, (\partial_n))$ of $\IncAlg$ modules and $M$ is an $\IncAlg$-module, is the complex 
\[\dots\to\Hom_{\IncAlg }(C_n, M) \xrightarrow{\partial_{n+1}^* }\Hom_{\IncAlg }(C_{n+1}, M) \to \dots\]
Note that we omit a conventional sign for the boundary map as it plays no role in our computations. This is a cochain complex as the functor $\Hom_{\IncAlg} (-, M)$ is contravariant. Assuming that $C_n = \bigoplus_{x\in S_n} P_x$ with $S_n$ a finite multi-subset of $ X$ and taking its cohomology gives shifted morphisms in the homotopy category $\HomotopyCat(\IncAlg)$ \cite[Lemma 3.7.10]{zimmermann2014representation}, which in turn are isomorphic to the shifted morphism in the derived category because the source is a perfect complex: 
\begin{equation}\label{eq:TotalHomHom}
\HH^{i}(\Hom^{\bullet}_{\IncAlg}(C, M)) \cong \Hom_{\HomotopyCat(\IncAlg)}(C, M[i])\overset{u}{\cong}\Hom_{D^b}(C, M[i]).
\end{equation}
Most computations will be carried out explicitly in the homotopy category. When needed the passage from one to the other will be discussed. Finally, using equation (\ref{eq:sqIota}) we have an isomorphism of cochain complexes 
\begin{equation}\label{eq:totHomComp}
\begin{tikzcd}
\dots\ar[r]&\Hom_{\IncAlg}(\bigoplus_{x\in S_n} P_x, M)\ar[r, "\partial_{n+1}^* "]\ar[d]&\Hom_{\IncAlg}(\bigoplus_{x\in S_{n+1}} P_x, M)\ar[r]\ar[d]&\dots\\
\dots\ar[r]&\bigoplus_{x\in S_n}e_xM\ar[r]& \bigoplus_{x\in S_{n+1}}e_xM\ar[r]&\dots
\end{tikzcd}
\end{equation}
The boundary maps of the bottom complex are linear combinations of left multiplication by elements $(x, y)$ of the algebra with coefficients inherited from the top complex.

%%%%%%%%%%%%%%%%%%%%%%%%%%%%%%%%%%%%%%%%%%%%%%%%%%%%%%%%%%%%%%
\paragraph{Antichain Modules} 
Let $(L, \wedge, \vee)$ be a finite lattice. We write $\hat{1}$ its greatest element and $\hat{0}$ its least one. For elements $a, b$ of $L$, let $[a, b] :=\{c\in L| a\leq c\leq b\}$ denote the interval of $L$. Let $C$ be an \emph{antichain} in $L$ \emph{i.e.} a subset $C$ of $L$ that consists of pairwise incomparable elements of $L$. We say an antichain $C$ is below an element $\alpha$ of $L$ if for all $c\in C$, we have $c\leq \alpha$, and when needed we record this information by the notation $C_{\alpha}$. Following \cite[Proposition 2.1]{iyama2022distributive} we associate to an antichain $C = \{c_1, \dots, c_r\}$ the submodule
$$N_C = \sum_{i=1}^{r} A\cdot(c_i,\hat{1})$$
of the projective indecomposable $P_{\hat{1}}$ generated by the antichain. It follows directly from the same proposition that there is a one to one correspondance between antichains and submodules of $P_{\hat{1}}$. The \emph{antichain module} associated to $C$ is defined by
$$M_C := P_{\hat{1}}/N_C.$$
We will talk of antichain modules below $\alpha\in L$ by restricting to the sublattice $[\hat{0}, \alpha]$ of $L$. Then $\alpha$ is the greatest element of this lattice and there is a bijection between submodules of $P_\alpha$ and antichains below $\alpha$. The corresponding modules will be denoted $N_C^{\alpha}$ and $M_C^{\alpha}$.
As our main example consider $a\leq b$ in $L$. The maxima of the set of elements of $L$ which are strictly less than $b$ but not above $a$ form an antichain $C$ and the antichain module below $b$ associated to $C$ has support the interval $[a, b]$. The corresponding antichain module is usually called an \defn{interval module}. In the rest of the paper we identify intervals with their interval modules.
\begin{lem}
Intervals are antichain modules.
\end{lem}
With the convention of the previous paragraph, morphisms between interval modules follow a simple rule
\begin{equation}\label{eq:morphIntS}
\Hom_{\IncAlg}([a, b], [c, d]) = \begin{cases}\field &\text{if } a\leq c \leq b\leq d,\\
									0 &\text{otherwise.} \end{cases}
\end{equation}
By  \cite[Theorem 2.2]{iyama2022distributive}, for every antichain $C$ of cardinal $r$ 
\ifBelowAlpha
below an element $\alpha$ of a lattice $L$ the associated antichain module $M^{\alpha}_C$ has a projective resolution $\PRes_{C}^{\alpha}$
$$0\to P_r \to \dots\to P_0 \to M_C^{\alpha} \text{ where } P_0 = P_{\alpha} \text{ and } P_l = \bigoplus_{\substack{S\subseteq C\\|S|=r}} P_{\wedge S}\text{ for } 1\leq l\leq r $$
\else
of a lattice $L$ the associated antichain module $M_C$ has a projective resolution $\PRes_{C}$ of the form
$$0\to P_r \to \dots\to P_0 \to M_C \text{ where } P_0 = P_{\hat{1}} \text{ and } P_l = \bigoplus_{\substack{S\subseteq C\\|S|=l}} P_{\wedge S}\text{ for } 1\leq l\leq r. $$ 
Similarly, define a resolution $\PRes_C^{\alpha}$ for the antichain module $M_C^{\alpha}$ below $\alpha$ by replacing $P_{\hat{1}}$ by $P_{\alpha}$.
\fi
The boundary maps are defined by fixing an arbitrary total ordering of elements in $C$ and, in each degree, setting the following maps between the indecomposable summands of the source and target in each degree:
\begin{align}\label{eq:boundComp}\begin{split}
P_{\wedge  S} \quad &\to\quad\quad\quad P_{\wedge T}\\
\big(x, \wedge S\big)\quad &\mapsto \quad\begin{cases}(-1)^{|i|_S}\big(x,\wedge T\big) &\text{if }T \sqcup \{i\} = S, \\ \quad \quad 0&\text{otherwise}\end{cases}
\end{split}\end{align}
for each $S = \{i_1, \dots, i_k\}$ and $\big(\wedge S,\wedge T\big)\in P_{\wedge T}$ where $|i|_S = |\{j\in S | j \leq i\}|$.
\subsection{Detailed outline}
In Section \ref{section4} we introduce four properties on antichains in lattices: \defn{intersectivity}, \defn{inclusivity}, \defn{strength} and \defn{booleanity}. They are related as in Figure \ref{fig:PropsOfAntichians}.  

\begin{wrapfigure}[12]{l}{0.4\textwidth}\centering
		\begin{tikzpicture}[squarednode/.style={rectangle, draw, fill=gray!5, very thick, minimum size=5mm}]
		\node[squarednode] (intersective) {\text{intersective}};
		\node[squarednode] (inclusive)  [below = of intersective] {\text{inclusive}};
		\node[squarednode] (strong)  [below = of inclusive] {\text{strong}};
		\node[squarednode] (boolean) [right = of inclusive] {\text{boolean}};
		\draw (-1.3, -3.7) rectangle  (1.3, .4) ;
		\node at (1.15, -1.6) (box) {};
		\draw[{implies}-{implies}, double] (box) --node[above]{def} (boolean);
		\draw[{implies}-{implies}, double] (inclusive) --node[left] {\ref{lem:InclusiveStrong}} (strong);
		\draw[-{implies}, double](intersective) -- node[left]{\ref{lem:Intersektstronk}}(inclusive);
		\end{tikzpicture}
	\caption{Properties of Antichains}\label{fig:PropsOfAntichians}
\end{wrapfigure}
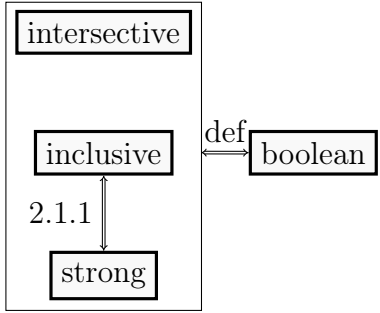 

In Lemma (\ref{lem:BoolBool}) we show that an antichain $C$ is boolean \emph{if and only if} it spans a boolean sublattice in $L$. This justifies the terminology. Boolean antichains have a crucial property (Theorem \ref{lem:homsBool}): hom spaces from a \defn{boolean} antichain module to an \defn{interval} are at most one dimensional and are concentrated in one degree. This does not hold for antichains which are only strong. However, certain hom spaces can still be controlled well enough. More specifically Lemmas \ref{lem:HoCatHom_Bool1} and \ref{lem:HoCatHom_Bool2} describe the maps between a resolution of a \defn{strong} antichain module and its truncations. Using these lemmas we prove Theorem \ref{claim:Gen} which is the main technical result of this paper.

	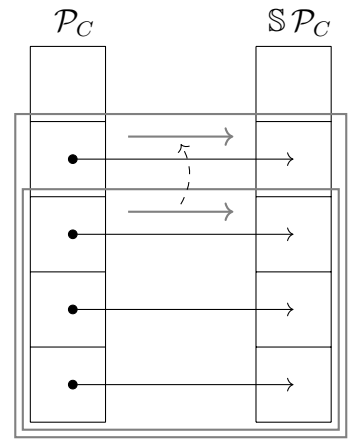
\begin{wrapfigure}[15]{r}{0.4\textwidth}\centering
	\vspace{-15pt}
	\begin{tikzpicture}
	\draw[step = 1cm](0,0) grid (1,5) node[above] {$\PRes_{C}\qquad$};
	\draw[step = 1cm](3,0) grid (4,5) node[above] {$\Serre^d \PRes_{C}[l]\qquad$};;
	\draw[thick, gray] (-0.1, -0.1) rectangle (4.1, 3.1);
	\draw[thick, gray] (-0.2, -0.2) rectangle (4.2, 4.1);
	\draw[->, thick, gray] (1.3, 2.8) -- (2.7, 2.8);
	\draw[->, thick, gray] (1.3, 3.8) -- (2.7, 3.8);
	\foreach \n in {.5, 1.5, 2.5, 3.5}{\draw[{Circle[]}->] (.5, \n) -- node[below left]{\tiny (OIH)} (3.5, \n);}
	\draw[->, dashed] (2, 2.9) to [bend right] node[right]{\tiny (II)} (2, 3.7);
	\end{tikzpicture}
	\caption{Isomorphisms of graded objects out of homogeneous one}\label{fig:keller}
	\end{wrapfigure}

It is a categorification theorem that builds upon and broadens \cite[Theorem~3.1]{rognerud2018bounded} which states that the Calabi--Yau property can be checked on the projective indecomposable modules in a finite poset with a least or greatest element. Our result implies that, for finite lattices, the property can be checked on any family of non zero strong antichain modules as long as it is sufficiently large. The proof consists in constructing isomorphisms $\Serre^{d}(P)\cong P[l]$ for projective indecomposable modules using the isomorphisms $\Serre^{d}(\PRes_C)\cong \PRes_C[l]$ on the family of antichain modules. We proceed by strong induction on the elements of the lattice (OIH). The main difficulty lies in constructing an isomorphism between graded objects out of ismorphisms in each degrees. This will take the form of a so called inner induction (II). See Figure \ref{fig:keller} for an illustration. The fact that these objects come from strong antichains will be crucial and Lemmas \ref{lem:HoCatHom_Bool1} and \ref{lem:HoCatHom_Bool2} will be used to ensure that the required squares commute by \defn{rectifying} any discrepancy. As result we can apply the axiom \textbf{TR3} of triangulated categories and the 2 out of 3 Lemma \ref{lem:2out3} to gain the isomorphism we want.
\par Theorem \ref{claim:Gen} will be applied in Section \ref{section2} to the incidence algebra of the lattice of order ideals of the product of two linearly ordered sets. To discuss antichains it is convenient to see the elements of this lattice as paths in grids. We consider a family of antichains introduced in \cite{Yildirim_2018}. Its associated objects are written $(\PRes_{\alpha})_{\alpha\in J_{m,n}}$. After recalling the main arguments of the proof from \cite{Yildirim_2018} and adapting them to our convention, we will show that these antichains satisfy the conditions of Theorem \ref{claim:Gen} and thus prove Theorem \ref{thm1}. 
\par In Section \ref{section3} we describe morphisms between the objects $\PRes_{\alpha}$ and their shifts. The goal of Section \ref{section3} is to prove Theorem \ref{thm:fCYTypeA}. We call $\mathcal{Y}_{m,n}$ the full subcategory of $D^{b}(\IncAlg)$ whose objects are the antichain modules $\PRes_{\alpha}$ and their shifts. Knowing that they are intervals and that their corresponding antichains are boolean  Theorem \ref{lem:homsBool} implies that each hom space is of dimension at most one. Proposition \ref{prop:factHomYildCat} gives the following canonical factorisation of morphisms in $\mathcal{Y}_{m,n}$ into extensions and degree zero morphisms.

\begin{equation}\label{eq:IntroHomDecomp}
\begin{tikzcd}[column sep = tiny]
\PRes_{\alpha}\ar[rr]\ar[dr]&&\PRes_{\beta}[|J|]\\
&\PRes_{q_J(\alpha)}[|J|]\ar[ur]&
\end{tikzcd}
\end{equation}
The extension is explicitly described in Proposition \ref{fact:PurExt} using transformations $q_J$ on partitions that appear in Definition \ref{def:qj}. The extension is further decomposed into \defn{irreducible extensions} in Lemma \ref{lem:decompExt}. The proof relies on computations in the homotopy category of complexes as the source of the morphism is identified with its projective resolution and the target is an interval. 
It follows from Proposition \ref{prop:factHomYildCat} that the degree one morphisms are parametrised by a subset of the original antichains themselves. It is also crucial to the proof that not all subsets of the antichains yield extensions. Definition \ref{def:allowedYild} provides a characterisation of these subsets which we call \defn{allowed} subsets. 

The second component of the composition in equation (\ref{eq:IntroHomDecomp}) is a morphism between intervals of the form $[f(\alpha), \alpha]$. Morphisms between intervals are described by equation (\ref{eq:morphIntS}) and comparison of partitions is done termwise. Most of the proofs amount thus to checking inequalities of the form:
\[f(\alpha)_i\leq f(\beta)_i \leq \alpha_i \leq \beta_i\]
for appropriate indices $i \leq m$. Proposition \ref{lem:hom0} gives alternative characterisations of morphisms between the objects $\PRes_{\alpha}$. Corollary \ref{nota:Hom0} highlights certain morphisms 
$\PRes_{\alpha}\to\PRes_{p_J(\alpha)}$ using the new characterisations and transformations $p_J$ on partitions that appear in Definition \ref{def:allowedDual}. Finally, Lemma \ref{lem:decomHom0} provides a decomposition into irreducible morphisms. 

Next, Proposition \ref{prop:RelYildCat} describes the relations between these morphisms. It uses Lemma \ref{lem:RelIdentification} which identifies the relations using the same manipulation of morphisms between intervals and complexes as before. In the process, in Proposition \ref{prop:configArrows} we interpret the morphisms in a combinatorial setting that links $\mathcal{Y}_{m,n}$ to \defn{Higher Auslander algebras of type A}. In Corollary \ref{cor:Pres} we give three slightly different yet convenient presentations of the category $\mathcal{Y}_{m,n}$ with generators and relations. This leads us to prove Theorem \ref{thm:fCYTypeA}. In Proposition \ref{prop:thicc} and Lemma \ref{lem:shiftTilting} we extract a tilting object from the category $\mathcal{Y}_{m,n}$.

\begin{equation}\label{eq:tiltings}
T: = \bigoplus_{\alpha\in J_{m,n}} \PRes_{\alpha}[i_{\alpha}]
\end{equation}

The integers $i_{\alpha}$ ensure that $T$ has no self extensions and are encoded using only the partitions $\alpha$. The fact that $\Thicc(T)$ generates the derived category can already be seen in the proof of Theorem \ref{claim:Gen} as every projective is obtained as a succession of cones from the family of antichains.
One of the presentations of Corollary \ref{cor:Pres} concludes the proof of Theorem \ref{thm:fCYTypeA} while another shows that $\End(T)^{op}$ is isomorphic to the \defn{quadratic dual} $(A_{n+1}^{m-1})^{!}$ of the higher Auslander algebra of type $A$. %Thus $\End(T)^{op}$ is an intermediate object between the Auslander algebra and its quadratic dual.

\section{Antichain modules and the Calabi--Yau property}
\label{section4}
This section contains technical results about certain classes of antichain modules, their morphisms and extensions. The main one gives a way to study the fractionally Calabi--Yau property on lattices. Once the correct antichains are identified, the proof is formal.
\subsection{Boolean antichains}
Let $C$ be an antichain of size $r$ in a lattice $L$ and  $M^{\alpha}_C$ its associated antichain module below $\alpha\in L$.
Note that in degree $i$ of the projective resolution $\PRes^{\alpha}_C$ of $M^{\alpha}_C$ there are $\binom{r}{i}$ indecomposable components in the direct sum.  If one forgets the modules, the complex has the shape of the power set of $C$, however the indices of the modules in each degree are not necessarily in bijection with the lattice $(\PRes(C), \subseteq, \cup, \cap)$ (see Figure \ref{fig:antichains1}). 
\begin{figure}[h]
	\begin{minipage}{0.3\textwidth}\centering
	%\vspace{17pt}
		\begin{tikzpicture}
			\draw (0,0)  -- (1,1) -- (1, 2) -- (0,3)--(-1,2)
			 --(-1, 1)--(0,0)--(0,1)--(-1, 2);
			\draw (-1, 1) -- (0,2) -- (1,1);
			\draw (0,1) -- (1, 2);
			\draw (0, 2) -- (0,3);
			\fill[red] (-1, 2)node["$c_1$" color = black]{} circle (2pt);
			\fill[red] (0, 2)node["$c_2$" color = black, left]{} circle (2pt);
			\fill[red] (1, 2)node["$c_3$" color = black]{} circle (2pt);
		\end{tikzpicture}
	%\vspace{3pt}
		\caption*{Boolean antichain}\label{fig:bool}
	\end{minipage}\hfill
	\begin{minipage}{0.4\textwidth}
	\hspace{-20pt}
		\begin{tikzpicture}
			\draw (0,0)  -- (1,1) -- (1, 2) -- (0,3)--(-1,2)
			 --(-1, 1)--(0,0)--(0,1)--(1, 2);
			\draw (-1, 1) -- (0,2) -- (1,1);
			\draw (0,1) -- (-1, 1.6);
			\draw (0, 2) -- (0, 3);
			\fill[red] (-1, 2)node["$c_1$" color = black]{} circle (2pt);
			\fill[red] (0, 2)node["$c_2$" color = black]{} circle (2pt);
			\fill[red] (1, 2)node["$c_3$" color = black]{} circle (2pt);
			\fill (-1, 1.6)node[label=left:$\substack{(c_2\wedge c_1)\\ \vee (c_3\wedge c_1)}$]{} circle (2pt);
			\draw (-1, 2) circle (3pt);
		\end{tikzpicture}\centering
		%\vspace{2pt}
		\caption*{Strong, not intersective, antichain}\label{fig:stronk}
	\end{minipage}\hfill
	\begin{minipage}{0.3\textwidth}
		%\vspace{32pt}
		\begin{tikzpicture}
			\draw (0,0) -- (-1, 1) -- (0,2) -- (0,1) -- (0,0) -- (1,1) -- (0,2);
			\fill[red] (-1, 1)node["$c_1$" color = black]{} circle (2pt);
			\fill[red] (0, 1)node["$c_2$\ \ " color = black, left]{} circle (2pt);
			\fill[red] (1, 1)node["$c_3$" color = black]{} circle (2pt);
			\fill (0, 0)node["$c_1\wedge c_2 = c_1\wedge c_3$" below]{} circle (2pt);
		\end{tikzpicture}\centering
		\vspace{10pt}
		\caption*{Antichain which is neither}\label{fig:nonBool}
	\end{minipage}
	\caption{Examples and non examples for key properties of antichains}\label{fig:antichains1}	
\end{figure}
To make this statement precise, let us introduce three conditions on $C$ as an antichain of $[\hat{0}, \alpha]$ for some $\alpha\in L$.
	\begin{description}
		\item[Inclusive antichain]For all subsets $S$ and $S'$ of  $C$, if $\wedge S \leq \wedge S'$ then $S'\subseteq S$.
		\item[Intersective antichain]For all subsets $S$ and $S'$ of  $C$, we have $(\wedge S) \vee (\wedge S') = \wedge(S\cap S')$.
		\item[Strong antichain]For all $S, S'$ subsets of $C$ of same cardinal, $\wedge S$ and $\wedge S'$ are incomparable \emph{i.e.} if $\wedge S \leq \wedge S'$ then $S = S'$.
		\item[Boolean antichain] C is both inclusive and intersective.
	\end{description}
Note that intersectivity depends on the choice of a top element $\alpha$ whereas inclusivity and strength do not. The meet and join operations, will be computed in the interval $[\hat{0}, \alpha]$. Note also the following lemma.
\begin{lem}\label{lem:InclusiveStrong}An antichain is {\bf inclusive} if and only if it is {\bf strong}.\end{lem}
\begin{proof}
The inclusivity condition implies the strong antichain condition by taking the subsets $S$ and $S'$ have the same cardinal. To see the converse, assume that the antichain  $C$ is a strong antichain and let $S$ and $S'$ be two subsets of $C$ such that $\wedge S' \leq \wedge S$. Suppose at first that $|S| +n = |S'|$ with $n > 0$. Then there exists $s_1, \dots, s_n \in S' \setminus S$. Set $S'' = S\sqcup \{s_1, \dots, s_n\}$.  Because the inequalities $\wedge S' \leq \wedge S $ and $\wedge S' \leq \wedge  \{s_1, \dots, s_n\}$ hold, we have 
\begin{equation*}
\wedge S' \leq (\wedge S)\wedge (\wedge\{s_1, \dots, s_n\}) = \wedge S''.
\end{equation*}
Because $|S'|=|S''|$, the antichain is strong, $S' = S''$ hence $S\subseteq S'$. Next if $|S| = |S'| + n$ with $n$ positive, then take $s_1, \dots, s_n$ in $S\setminus S'$. Then we have 
\[\wedge S'' := (\wedge S')\wedge (\wedge\{s_1, \dots, s_n\})\leq \wedge S\]
The antichain is strong so $S = S''$. Then $S' \subseteq S$ so $\wedge S' \geq \wedge S$ and thus $\wedge S' = \wedge S$. Using the first part of the proof we get $S = S'$. This contradicts the assumption on the integer $n$.
\end{proof}
\begin{rem}The strong antichain condition implies that for each $n$, the set \[\{\wedge S | S\subseteq C \text{ with } |S| = n\}\] is an antichain. This condition is strong enough for the main result (see Subsection \ref{sec:mainResult}). However it is not strong enough for  computing morphism sets.
\end{rem}
\begin{lem}\label{lem:Intersektstronk}Suppose $C$ is an intersective antichain below an element $\alpha\in L$ and suppose that it is not the singleton $\{\alpha\}$. Then $C$ is strong.
\end{lem}
\begin{proof} Suppose the antichain $C$ is intersective. If it contains only one element, then it is clear that it is strong. Suppose it contains at least two elements. Let $S$ and $S'$ be two subsets of the antichain $C$ such that $\wedge S \leq \wedge S'$. Let $x$ be an element of $S'$. Then  we have $x \geq \wedge S' \geq \wedge S$ and also $(\wedge S) \vee x = \wedge (S\cap \{x\})$. Because $x\not = \hat{1}$, it follows that $x\in S$ which concludes the argument. 
\end{proof}
\begin{cor}An antichain below $\alpha$ which is not $\{\alpha\}$ is boolean if and only if it is intersective.
\end{cor}

Denote $\langle C\rangle_{\vee, \wedge}^{\alpha}$ the lattice generated by the elements of $C$ and $\alpha$ in the sublattice $[0, \alpha]$ of $L$, equipped with the lattice operations of $L$. The following lemma motivates the terminology.
\begin{lem}\label{lem:BoolBool}An antichain is boolean if and only if the map 
\begin{align*}
(\mathcal{P}(C), \cap, \cup)&\xrightarrow{\phi} (\langle C\rangle_{\vee, \wedge}^{\alpha},  \wedge, \vee)\\
S\quad\quad &\mapsto \quad\quad \wedge S
\end{align*}
is a lattice anti-isomorphism.
\end{lem}
\begin{proof} 
Assume that the map $\phi$ is a lattice anti-isomorphism. Then $C_{\alpha}$ is intersective because $\phi$ sends $\cap$ to $\vee$. Conversely assume $C$ is both intersective and thus inclusive below $\alpha$. The fact that $\phi$ sends $\cup$ to $\wedge$ is true for any subset of a lattice. The intersection property makes $\phi$ send $\cap$ to $\vee$. To see that $\phi$ is injective, note that if $\wedge S = \wedge S'$ then the inclusion property forces $S = S'$ for any $S, S'\subseteq C$. Next notice that the image of $\phi$, $\Ima(\phi) = \{\wedge S | S\subseteq C\}$ is a lattice, using the properties we just exhibited. Moreover, any sublattice of $[\hat{0}, \alpha]$ containing $C$ contains $\Ima(\phi)$. It is thus the sublattice of $[\hat{0}, \alpha]$ generated by  $C$ and $\alpha$, \emph{i.e.}, $\langle C\rangle_{\vee, \wedge}^{\alpha} = \{\wedge S | S\subseteq C\}$ and $\phi$ is surjective.
%Assume that the map $\phi$ is a lattice anti-isomorphism. Then $C_{\alpha}$ is intersective because $\phi$ sends $\cap$ to $\vee$. Now consider $S, S'\subseteq C$ such that $\wedge S \leq \wedge S'$. This is equivalent to the following equality 
%$$\wedge S = (\wedge S) \wedge (\wedge S').$$
%The right hand side is equal to $\wedge (S\cup S')$. Because $\phi$ is a bijection, $S = S\cup S'$ meaning that $S'\subseteq S$. Thus $C$ is inclusive. Conversely assume $C$ is both inclusive and intersective below $\alpha$. The fact that $\phi$ sends $\cup$ to $\wedge$ is true for any subset of a lattice. The intersection property makes $\phi$ send $\cap$ to $\vee$. To see that $\phi$ is injective, note that if $\wedge S = \wedge S'$ then the inclusion property forces $S = S'$. To see that the map is surjective, notice that the image of $\phi$, $\Ima(\phi) = \{\wedge S | S\subseteq C\}$ is a lattice, using the properties we just exhibited. Moreover, any sublattice of $[\hat{0}, \alpha]$ containing $C$ contains $\Ima(\phi)$. It is thus the sublattice of $[\hat{0}, \alpha]$ generated by  $C$ and $\alpha$, \emph{i.e.}, $\langle C\rangle_{\vee, \wedge}^{\alpha} = \{\wedge S | S\subseteq C\}$ and $\phi$ is surjective.
\end{proof}
%%%%%%%%%%%%%%%%%%%%%%%%%%%%%%%%%%%%%%%%%%%%%%%%%%%%%%%%%%%%%%%%%%%%%%%%%%%%%%%%%%%%%%%%%%%%%%%%%%%%%%%%%%%%%%%%%%%%%%%%%%%%
\subsection{Morphisms}
In this subsection we fix an antichain $C$ below an element $\alpha$ of a lattice $L$ as well as an interval $I = [a, b]$ of $L$.
If $C$ is \defn{boolean} more can be said about the total hom complex $\Hom^{\bullet}_{\IncAlg}(\PRes_C^{\alpha}, I)$. Recall the notation from equation (\ref{eq:totHomComp}). Note that 
\begin{equation}\label{eq:morphInt}
e_x\cdot I = \begin{cases}\field &\text{if } x\in I \\0&\text{otherwise}\end{cases}
\end{equation}
as a special case of equation (\ref{eq:morphIntS}). Denote by $E$ the set $\{S\subseteq C | \wedge S\in I\}$. A projective module $P_{\wedge S}$ appearing in the projective resolution $\PRes_{C}^{\alpha}$ can contribute to the total hom if and only if $S \in E$. Note that $C$ is finite hence its set of subsets is finite. Assume that $E$ is not empty otherwise the total hom complex is zero. If $C$ is intersective, and if $S$ and $S'$ are in $E$ then 
\[a\leq \wedge(S\cup S')\leq \wedge (S\cap S')\leq b\]
meaning that $S\cup S'$ and $S\cap S'$ are in $E$ as well. It follows that $E$ has a largest element $\cup_{S\in E} S = S_M$ as well as a least element $\cap_{S\in E} S = S_m$. Moreover, if $S \in [S_m, S_M]$ then \[a\leq \wedge S_M\leq \wedge S \leq \wedge S_m\leq b\] hence $S\in E$ and we have proved the following lemma.
\begin{lem}\label{lem:Eint} If $C$ is a boolean antichain then $E = [S_m, S_M]$. \end{lem}
To describe the total hom complex, $\Hom^{\bullet}_{\IncAlg}(\PRes_C^{\alpha}, I)$, we write $m = |S_m|$ and $M = |S_M|$. For each degree $m\leq i\leq M$  there are exactly $\binom{M-m}{i - m}$ subsets $S$ of $C$ with cardinal $i$ in $E$. Hence, by equations ((\ref{eq:totHomComp} and \ref{eq:morphInt})) the complex $\Hom^{\bullet}_{\IncAlg}(\PRes_C^{\alpha}, I[0])$ has shape:
\begin{equation}\label{eq:simplicialComp}
0\leftarrow \field^1 \leftarrow \dots \leftarrow \field^{\binom{M-n}{j}} \leftarrow \dots \leftarrow \field\leftarrow 0.
\end{equation}
This is precisely the shape of the simplicial module associated to the standard simplex. It remains to describe that the boundary maps match the standard simplex as well. The map is post composition by the boundary map of $\PRes_C^{\alpha}$ in degree $i$. It sends a map defined by a vector $(f_S)_{|S|= i}$ to a vector $(g_{S'})_{|S'| = i+1}$ described by
\begin{equation*}
g_{S'} = \begin{cases} 
		0 &\text{if } S\not \subseteq S'\\ 
		(-1)^{\epsilon(S, x)}\cdot f_S &\text{otherwise,}\end{cases}
\end{equation*}
where $|x|_{S} = |\{y\in S' | y\leq x\}|$. Indexing the vector elements by their complements in $C$ and writing the basis vectors $e_J$ we get 
\begin{equation*}
e_J \mapsto \sum_{x\in J} (-1)^{|x|_{J^{c}}} e_{J-\{x\}}
\end{equation*}
Like in the proof of \cite[Theorem 2.2]{iyama2022distributive}, we have an isomorphism
\begin{equation}\label{eq:isoKoszulComplex}
\Hom^{\bullet}_{\IncAlg}(\PRes_C^{\alpha}, I)\cong (\field \xleftarrow{id} \field)^{\otimes (M-m)}[m].
\end{equation}
The left hand side is a Koszul complex over the base field $\field$. As a tensor product of acyclic complexes it is either acyclic or concentrated in one degree when the set $S_M\setminus S_m$ is empty using K\"unneth's formula \cite[Chapter 6.3]{Cartan1956Homological} or \cite[Exercise 1.2..5]{weibel1994introduction}.
\begin{prop}\label{prop:homInt} Let $C$ be an antichain of a lattice $L$ and let $I \subseteq L$ be an interval. Suppose the set $E = \{S\subseteq C | \wedge S\in I\}$ is an interval of the lattice $\PRes(C)$. Then there exists at most one integer $p$ such that $\Hom_{\HomotopyCat}(M_C, I[p])$ is non zero. When such an integer exists, the hom set is one dimensional.
\end{prop}
\begin{proof} Given the previous calculations and remarks, the result follows from equation (\ref{eq:TotalHomHom}).
\end{proof}
 Moreover we know exactly that such a degree exists if and only if the set $E$ is a singleton \emph{i.e.} there exists a unique $S\subseteq C$ such that $\wedge S \in I$. In this case $p = |S|$. 
\begin{thm}\label{lem:homsBool}
 Let $C$ be a \defn{boolean} antichain of a lattice $L$ \ifBelowAlpha below an element $\alpha\in L$\fi. Let $I \subseteq L$ be an interval. There exists at most one integer $p$ such that $\Hom_{D^b}(M_C, I[p])$ is non zero. When such an integer exists, the hom space is one dimensional.
\end{thm}
\begin{proof}
This follows from Proposition \ref{prop:homInt} combined with Lemma \ref{lem:Eint} and the isomorphism $u$ of equation (\ref{eq:TotalHomHom}).
\end{proof}
%%%%%%%%%%%%%%%%%%%%%%%%%%%%%%%%%%%%%%%%%%%%%%%%%%%%%%%%%%%%%%
%%%%%%%%%%%%%%%%%%%%%%%%%%%%%%%%%%%%%%%%%%%%%%%%%%%%%%%%%%%%%%
\subsection{Truncations}
The stupid truncation \cite[Remark 3.5.22]{zimmermann2014representation}  $\sigma_{\geq i} R$ of a complexe $R = ((R_n)_{n\in \mathbb{Z}}, (\partial_n)_{n\in\mathbb{Z}})$ is defined by 
\begin{equation}
(\sigma_{\geq i} R)_n = \begin{cases} R_n &\text{ if } n \geq i\\ 0&\text{otherwise}
					\end{cases}
\text{ with boundaries }
\partial'_n = \begin{cases} \partial_n &\text{ if} n > i\\ 0&\text{ otherwise. }
					\end{cases}
\end{equation}

In this subsection we fix a \defn{strong} antichain $C$. To make notation lighter we do not say if it is below some $\alpha$, though the two lemmas below hold in both $[\hat{0}, \alpha]$ and $L$.
\begin{lem} \label{lem:HoCatHom_Bool1}
With $C$ as above and notation from the previous subsections, for all $r = |C| \geq i\geq 0$ there is a bijection
 \begin{align*}
\Phi : \End_{\IncAlg} ((\PRes_C)_{i-1}) &\to \Hom_{\HomotopyCat}(\sigma_{\geq i}\PRes_C, (\PRes_C)_{i-1}[i])\\
 f\quad\quad & \mapsto \quad\quad f\circ \partial_i [i]. 
 \end{align*}
\end{lem}
 \begin{proof}
If $i = 0$, $\PRes_{i-1} = 0$. So we assume $r \geq i > 0$. The antichain is strong so the indices of the indecomposable summands of $(\PRes_C)_{i-1} = \bigoplus_{|S| = i-1} P_{\wedge S}$ cannot be compared. Thus the endomorphisms of this module decompose as 
\begin{equation}\label{eq:endP}f = \bigoplus_{|S| = i-1} \lambda_S\cdot id_{P_{\wedge S}}.\end{equation}
By projecting on the summands of the target, it suffices to show that there is a bijection 
\begin{align*}
\Phi : \End (P) &\to \Hom_{\HomotopyCat}(\sigma_{\geq i}\PRes_C,P[i])  \\
 f\quad\quad & \mapsto \quad\quad f\circ \pi_P\circ \partial_i [i] 
  \end{align*}
for all $P = P_{S}$ with $S\subseteq C$ of cardinal $i-1$. Write $S = \{s_1, \dots, s_{i-1}\}$.
 The morphisms on the right hand side are of the form:
\begin{equation}\label{diag:trunc1}
\begin{tikzcd}
\bigoplus_{|S'| = i+1}P_{\wedge S'}\ar[r, "\partial_{i+1}"]\ar[d, "\dots\quad\quad"']& \bigoplus_{|S'| = i}P_{\wedge S'}\ar[r]		\ar[d, "\phi"]\ar[dl, dashed, red, "0"]&0\ar[dl, dashed, red, "0"]\\
	0\ar[r]& P\ar[r]&0
\end{tikzcd}
\end{equation}
The red arrows represent potential homotopy maps. There cannot be a non zero homotopy so
\[ \Hom_{\HomotopyCat}(\sigma_{\geq i}R,P[i]) \cong  \Hom_{C}(\sigma_{\geq i}R,P[i]).\]
The relation $\partial^2 = 0 $ implies that $\Phi$ is well defined. An element of $\End(P)$ is either $0$ or an automorphism. Moreover, considering the specific form of the boundary maps of the complex $\PRes_C$ in equation (\ref{eq:boundComp}), the projection of $\partial_i$ on the factor $P$ of its target is non zero. Hence if $f$ is non zero then $\Phi(f)$ is non zero and $\Phi$ is injective. It remains to see that it is surjective.  

Notice that there is an isomorphism between $\Hom_{C}(\PRes_{C}, P[i])$ and  $\Hom_{C}(\sigma_{\geq i}\PRes_{C}, P[i])$ as any map $\phi$ as in equation (\ref{diag:trunc1}) also yields a map from the untruncated complex to $P[i]$ and vice versa. The map $\Phi_S$ is surjective if and only if every such map factors through $\partial_{i}$ \emph{i.e.} if and only if every element of  $\Hom_{C}(\PRes_{C}, P[i])$ is zero homotopic. Because the antichain is strong we have
\[E = \{S'\subset C| \wedge S'\leq \wedge S\} = \{S'\subseteq C | S \subseteq S' \}\]
where $E$ is the set of contributing subsets of $C$ in the total hom. The assumptions on $i$ and the cardinal of $C$ ensures that $E$ contains at least two elements. By Proposition \ref{prop:homInt} we have
\[\Hom_{\HomotopyCat}(\PRes, P_{\wedge S}[i]) \cong 0.\]This concludes the proof.
 \end{proof}
%%%%%%%%%%%%%%%%%%%%%%%%%%%%%%%%%%%%%%%%%%%%%%%%%%%%%%%%%%%%%%
%%%%%%%%%%%%%%%%%%%%%%%%%%%%%%%%%%%%%%%%%%%%%%%%%%%%%%%%%%%%%%
\begin{lem}\label{lem:HoCatHom_Bool2}
Let $C$ be a strong antichain. Then 
\begin{equation*}
\dim \Hom_{\HomotopyCat}(\PRes_C, \sigma_{\geq 1} \PRes_C) \leq 1.
\end{equation*}
\end{lem}
\begin{proof}If $r = |C| = 0$, then $\sigma_{\geq 1}\PRes_C = 0$ and the space of morphisms in question is 0 dimensional. Assume that $r$ is strictly bigger than 0. The setting of the lemma can be illustrated by the following diagram
\begin{equation*}
\begin{tikzcd}
\dots\ar[r] & \bigoplus_{s, t\in C} P_{s\wedge t}\ar[d]\ar[r]& \bigoplus_{s\in C} P_S\ar[d]\ar[r]\ar[dl, dashed, red, "0"]  & P_{\hat{1}}\ar[d]\ar[dl, dashed, red, "0"] \\
\dots\ar[r] & \bigoplus_{s, t\in C} P_{s\wedge t}\ar[r] &\bigoplus_{s\in C} P_S\ar[r] & 0
\end{tikzcd}
\end{equation*}
The antichain $C$ is inclusive so the projective indecomposables in degree $i$ are associated to elements which are either bigger than the ones in degree $i+1$ or cannot be compared with them. Hence there are no non zero maps of degree $1$. We thus need to describe the maps between $\PRes_{C}$ and $\sigma_{\geq 1}\PRes_C$ in the category of complexes. For $k\in \llbracket 1, r \rrbracket$ both complexes have the same components
so morphisms of complexes are determined by morphisms of modules 
\[\phi_k \in \End\big( \bigoplus_{\substack{S\subseteq C,\\ |S| = k}}P_{\wedge S}\big)\]
satisfying the relation 
\begin{equation}\label{eq:relMorph}
\phi_k\circ\partial_{k+1} = \partial_{k+1}\circ\phi_{k+1}.
\end{equation}
Since $C$ is a strong antichain, the elements $\wedge S$ with a fixed cardinal cannot be compared. Hence an  endomorphism $\phi_k$ of this module is of the form
\begin{equation}\label{eq:morphDecomp}
\phi_k = \bigoplus_{|S| = k} \lambda_S\cdot id_{P_{\wedge S}}
\end{equation}
with $\lambda_S\in\field$. The goal is to show that for $1\leq k\leq r$ we have $\phi_k = \lambda_{C}\cdot id_k$. In other words, for all $S$, $S'$ subsets of $C$, $\lambda_S = \lambda_{S'}$. We proceed by downward induction on $k$ starting with $k = r$. In that case, we already have $\phi_r = \lambda_C id_{P_{\wedge C}}$ as there is only one projective indecomposable summand in degree $r$. If $r = 1$, we are done. Now take $1 \leq k <r$ and assume $\phi_{k+1} = \lambda_C\cdot id_{k+1}$. We now put together equations (\ref{eq:boundComp}), (\ref{eq:relMorph}) and (\ref{eq:morphDecomp}). On the one hand we have
\begin{equation*}\partial_{k+1}\circ \phi_{k+1} = \lambda^r_C \cdot \partial_{k+1}\end{equation*}
Evaluating at $e_{\wedge S}\in P_{\wedge S} \subseteq \bigoplus_{|S| = k+1} P_{\wedge S}$.
\begin{equation}
	\partial_{k+1}\circ\phi_{k+1}(e_{\wedge S}) = \lambda_C\cdot \sum_{s\in S} (-1)^{|s|_S}\cdot e_{\wedge (S\setminus \{i\})}.
\end{equation}
On the other hand we have 
\[\phi_k\circ\partial_{k+1}(e_{\wedge S}) = \sum_{s\in S} (-1)^{|s|_S}\lambda_{S\setminus\{s\}}\cdot e_{\wedge (S\setminus\{s\})}. \]
Because the $e_{\wedge (S\setminus \{s\})}$ are linearly independent we get
\[\lambda_{(S\setminus \{s\})} = \lambda_C\]
for all $S\subset C$ of cardinal $k+1$ and $s\in S$. Noticing that any subset of $C$ of cardinal $k < r$ can be expressed as $S\setminus\{s\}$ for some $S$ and some $s$, this finishes the proof.
\end{proof}
%%%%%%%%%%%%%%%%%%%%%%%%%%%%%%%%%%%%%%%%%%%%%%%%%%%%%%%%%%%%%%
%%%%%%%%%%%%%%%%%%%%%%%%%%%%%%%%%%%%%%%%%%%%%%%%%%%%%%%%%%%%%%
\subsection{Main result}\label{sec:mainResult}

\begin{thm}\label{claim:Gen} 
Let $L$ be a finite lattice, let $d$ and $l$ be integers. Suppose there exists a family of antichains $(C_{\alpha})_{\alpha\in L}$ of L such that for all $\alpha\in L$, the following assumptions hold.
\begin{enumerate}
\item \label{ass:1}The antichain $C_{\alpha}$ is below $\alpha$.
\item \label{ass:2}The module $M^{\alpha}_C$ is non zero and there is an isomorphism
	\begin{equation}\label{eq:CY}
	\Serre^l M_{C_{\alpha}}^{\alpha}\cong M_{C_{\alpha}}^{\alpha}[d]
	\end{equation}in $D^b(\mathcal{A})$.
\item \label{ass:3}The antichain $C_{\alpha}$ is \defn{strong}.
\end{enumerate}
Then $D^b(\IncAlg)$ is $\frac{d}{l}$-Calabi--Yau.
\end{thm}
In practice, assumption (2) is the hardest to investigate. The proof relies on the following theorem.
\begin{thm}[\protect{\cite[Theorem~3.1]{rognerud2018bounded}\label{thm:baptiste}}]Let $X$ be a finite poset with a unique minimal or unique maximal element. If there are integers $d$ and $l$ such that $\Serre^l(P)\simeq P[d]$ for all projective indecomposable modules of $X$, then the category $D^b(\IncAlg_{\field}(X))$ is $\frac{d}{l}$\emph{-fractionally Calabi--Yau}.
\end{thm}
We also use the axioms \textbf{TR2} and \textbf{TR3} of triangulated categories \cite[Definition~3.4.1]{zimmermann2014representation} and this classical \emph{two-out-of three} lemma from \cite[Lemma~3.4.10]{zimmermann2014representation}.
\begin{lem}\label{lem:2out3}
Let $\mathcal{T}$ be a triangulated category with self equivalence $T$ and let 
	\begin{equation}
	\begin{tikzcd}
	A_1\ar[d]\ar[r]&A_2\ar[d]\ar[r]&A_3\ar[d]\ar[r]&TA_1\ar[d]\\
	B_1\ar[r]&B_2\ar[r]&B_3\ar[r]&TB_1
	\end{tikzcd}
	\end{equation}
be a morphism between distinguished triangles. If two of the vertical morphisms are isomorphisms, then the third one is an isomorphism as well.
\end{lem}
\begin{nota}
To simplify notation we now set $F = \Serre^{l}[-d]$. We will often refer to equation (\ref{eq:CY}) as the Calabi--Yau property for a certain complex of $\IncAlg$ modules.
\end{nota}
\begin{proof}[Proof of Theorem \ref{claim:Gen}]
By Theorem (\ref{thm:baptiste}), it suffices to prove that the Serre functor satisfies the Calabi--Yau property on the projective indecomposable modules. We proceed by induction on the elements $\alpha$ of the lattice $L$. For the initial step notice that, by the first assumption of the claim, the indecomposable projective module associated to the minimum of the poset $P_{\hat{0}}[0] \simeq M^{\hat{0}}_{C_{\hat{0}}}[0]$ satisfies equation (\ref{eq:CY}). Hence fix an element $\alpha > \hat{0}$ and assume as the induction hypothesis, that equation (\ref{eq:CY}) is true for all $\alpha' < \alpha$. We will refer to this as the \defn{outer induction hypothesis} (OIH). Write $C = C_{\alpha}$ for the antichain indexed by $\alpha$ and let $R = \PRes_{C}^{\alpha}$ be its associated projective resolution. If $r = |C| = 0$, then $R$ is a projective module concentrated in degree $0$ and there is nothing to prove by the second assumption in Theorem \ref{claim:Gen}. 
%%%%%%%%%%%%%%%%%%%%%%%%%%%%%%%%%%%%%%%%%%%%%%%%%%%%%%%%%%%%%%
\noindent\emph{Step1.} So assume  $ r > 0$ and consider the distinguished triangle
	\begin{equation}\label{tikzcd:trig1}
	\begin{tikzcd}
	P_{\alpha}[0]\ar[r]&R\ar[r]&\sigma_{\geq 1}R\ar[r, "f"]& P_{\alpha}[1]
	\end{tikzcd}\end{equation}
induced by the truncation short exact sequence \cite[\href{https://stacks.math.columbia.edu/tag/0118}{Section 12.15}]{stacks-project}. A computation shows that the map $f$ is the boundary map $\partial_1$ in degree 1 and zero in all other degrees. With \textbf{TR2} this triangle can be shifted to 
	\begin{equation}\label{tikzcd:trig2}
	\begin{tikzcd}
	R\ar[r]&\sigma_{\geq 1}R\ar[r]&P_{\alpha}[1]\ar[r]&R[1]
	\end{tikzcd}\end{equation}
Because $F$ is triangulated, it sends this triangle to a triangle and we have the following diagram.
	\begin{equation*}
	\begin{tikzcd}
	R\ar[r]\arrow{d}[anchor=center, rotate=-90,yshift=0.5ex]{\sim}&\sigma_{\geq 1}R\ar[r]\ar[d, dashed, "\exists ?" ]&P_{\alpha}[1]\ar[r]&R[1]\arrow{d}[anchor=center, rotate=-90,yshift=0.5ex]{\sim}\\
	F(R)\ar[r]&F(\sigma_{\geq 1}R)\ar[r]&F(P_{\alpha}[1])\ar[r]&F(R)[1]
	\end{tikzcd}
	\end{equation*}
If we can show that there exists an isomorphism following the dashed arrow and that the left square can be chosen to commute, we can apply \textbf{TR3} to complete the diagram followed by Lemma \ref{lem:2out3} to finish the proof.
%%%%%%%%%%%%%%%%%%%%%%%%%%%%%%%%%%%%%%%%%%%%%%%%%%%%%%%%%%%%%%

\noindent\emph{Step 2.} To construct an isomorphism 
\[\sigma_{\geq 1}R\xrightarrow{\sim}F(\sigma_{\geq 1 }R)\]
we show by downward induction on $i$ that
\[\sigma_{\geq i}R\xrightarrow{\sim}F(\sigma_{\geq i }R)\]
for $i \in \{1, \dots, r \}$. We refer to this as the \defn{inner induction hypothesis} (IIH). Taking $i = r$, the initial step follows from the outer induction hypothesis as $P_{\wedge C}[r] \simeq \sigma_{\geq r} R$. Now, fix $1\leq i < r$ and assume the property is true for $i$. To show that it is also true for $i-1$, like equation (\ref{tikzcd:trig1}), consider the truncation triangle
	\begin{equation*}
	\begin{tikzcd}
	\sigma_{\geq i}R\ar[r, "\partial_i{[i]}"]&R^{i-1}[i]\ar[r]&\sigma_{\geq i-1}R[1]\ar[r]&\sigma_{\geq i}R[1]
	\end{tikzcd}
	\end{equation*}
shifted using \textbf{TR2}. Again, we want to conclude by using \textbf{TR3} and Lemma \ref{lem:2out3} so it suffices to show that we can choose the vertical isomorphisms of the following square such that it commutes. 
	\begin{equation}\label{tikzcd:square}
	\begin{tikzcd}
	\sigma_{\geq i}R\ar[r, "\partial_i{[i]}"]\arrow{d}[anchor=center, rotate=-90,yshift=0.5ex]{\sim}\ar[d, "f"']& R^{i-1}[i]\arrow{d}[anchor=center, rotate=-90,yshift=0.5ex]{\sim}\ar[d, "g"']\\
	F(\sigma_{\geq i}R)\ar[r,  "F(\partial_i{[i]})"]& F(R^{i-1} [i])
	\end{tikzcd} 
	\end{equation}
By the inner induction hypothesis we can choose an isomorphism $f$ for the left vertical arrow and by the outer induction hypothesis we can choose an isomorphism $g$ for the right vertical arrow.
%%%%%%%%%%%%%%%%%%%%%%%%%%%%%%%%%%%%%%%%%%%%%%%%%%%%%%%%%%%%%%

\noindent\emph{Step 3.} The isomorphism $g$ can be rectified to make the square commutative. We compare the morphism $\partial_i$, with $g^{-1}\circ F(\partial_i{[i]})\circ f$. 
Recall the isomorphism $u$ between the morphism set in the homotopy and bounded derived category, provided that the source is a complexe of projectives. The isomorphism is functorial in the target and is defined by sending a class of morphisms up to homotopy $c$ to the morphism
\begin{equation*}
		\begin{tikzcd}[row sep=small]
		&C&\\
		P\ar[ur, "c"]&&C\ar[ul,"\sim" labl, "id"']
		\end{tikzcd}
	\end{equation*}
in the derived category. This is the "fraction" $\frac{c}{1}$. In the homotopy category we want to compare the morphism $\partial_i$, which is already a morphism of complexes, with $u^{-1}(g^{-1}\circ F(\partial_i{[i]})\circ f)$. By Lemma \ref{lem:HoCatHom_Bool1} there exists $T$ an endomorphism of $\bigoplus_{|S| = i-1} P_{\wedge J}$ such that 

\begin{equation}\label{eq:Tpartial}( T \circ \partial_i) [i]= u^{-1}( g^{-1}\circ F(\partial_i{[i]})\circ f).\end{equation}
Applying $u$ again we get
\begin{equation*}
\frac{T[i]}{1} \circ \frac{\partial_i[i]}{1}= g^{-1}\circ F(\partial_i{[i]})\circ f.\end{equation*}
From now on all the morphisms will be in the derived category but we omit the denominator of maps coming from the homotopy category as it is always equal to one.

%%%%%%%%%%%%%%%%%%%%%%%%%%%%%%%%%%%%%%%%%%%%%%%%%%%%%%%%%%%%%%
\par We prove that the map $T$ is an isomorphism. It has form $\bigoplus_{S} \lambda_S\cdot id_{P_{\wedge S}}$ because the antichain $C$ is strong. It is enough to show that the projections of $T$ onto the indecomposable summands of its target are non zero. Consider the following diagram.
	
	\begin{equation}\label{tikzcd:3.5squares}
	\begin{tikzcd}[column sep = large]
	&R^{i-1}[i]\ar[d, "T{[i]}"]\ar[r, "\pi_S{[i]}"]&P_{\wedge S}[i]\ar[d, "\textcolor{red}{\lambda_S}\cdot id{[i]}"]&\text{Shifted module maps}\ar[dll, red, shorten >=10ex, bend right = 10]\\
	\sigma_{\geq i}R\ar[ur, "\partial_i{[i]}"]\ar[d, "f", "\text{IIH}"' blue]& R^{i-1}[i]\ar[d, "g", "\text{OIH}"' blue]\ar[r, "\pi_S{[i]}"]&P_{\wedge S}[i]\ar[d, "h\leadsto \tilde h", "\text{OIH}"' blue]&\dim\Hom (R^{i-1}[i], FP_{\wedge S}[i]) = 1\ar[dll, red, shorten >=8ex, bend right = 10]\\
	F\sigma_{\geq i}R\ar[r, "F(\partial_i{[i]})"]\ar[uur, phantom, "\text{by } \ref{lem:HoCatHom_Bool1}" , blue]& FR^{i-1} [i]\ar[r, "F(\pi_S{[i]})"]&FP_{\wedge S}[i]&
	\end{tikzcd} 
	\end{equation}
The pentagon on the left is commutative since 
\[g\circ T[i]\circ \partial_i[i] = g\circ g^{-1}\circ F(\partial_i[i])\circ f =  F(\partial_i[i])\circ f.\]
The top right square illustrates equation (\ref{eq:endP}) so it commutes as well. For the bottom right square, the outer induction hypothesis gives  the isomorphism $h$. Next, $F$ is fully faithful so it induces an injective linear map on hom spaces. Hence $F(\pi_S[i])$ is non zero. Finally, $g$ is an isomorphism so the composition $F(\pi_{\wedge S}[i])\circ g$ is non zero too. Because the antichain is strong, 
\[\dim \Hom_{\IncAlg}(R^{i-1}, P_{\wedge S} [i]) = 1\]
and because $P_{\wedge S}[i]$ and $ FP_{\wedge S} [i]$ are isomorphic, we also have 
\[\dim \Hom_{\IncAlg}(R^{i-1}, FP_{\wedge S} [i]) = 1.\]
Hence it is possible to replace the map $h$ by $\tilde h = \lambda\cdot h$ for some $\lambda$ making it commute. Chasing around the diagram we thus compute
\begin{equation*}
\lambda_S\cdot \tilde{h} \circ \pi_S[i]\circ\partial[i] = F(\pi_S\circ\partial_i[i])\circ f.
\end{equation*}
By construction the projections of $\partial_i$ are all non zero and $\tilde{h}$ and $f$ are isomorphisms. Again, $F$ is an equivalence of categories so  the right hand side is non zero. Hence $\lambda_S\not = 0$ for all $S\subseteq C$ of cardinal $i$ and $T$ is an automorphism. Rearranging equation (\ref{eq:Tpartial}) we get
\begin{equation*}
\partial_i = T^{-1}\circ g^{-1}\circ F(\partial_i{[i]})\circ f.
\end{equation*}
So replacing $g$ by $g\circ T$ the square (\ref{tikzcd:square}) becomes commutative. Applying Lemma (\ref{lem:2out3}) and the inner induction hypothesis we have$$F(\sigma_{\geq 1}R)\cong \sigma_{\geq1}R.$$ 
%%%%%%%%%%%%%%%%%%%%%%%%%%%%%%%%%%%%%%%%%%%%%%%%%%%%%%%%%%%%%%
\emph{Step 4.} To complete the outer induction  choose such an isomorphism $g$, and an isomorphism $f: R \to F(R)$. Just like in step 3 above, we want to compare the map $p: R \to \sigma_{\geq 1}R$ with $g^{-1}\circ F(p)\circ f$ and, if needed, rectify the isomorphisms $g$ and $f$ to make the square commute. Lemma \ref{lem:HoCatHom_Bool2} implies that $g^{-1}\circ F(p)\circ f = \lambda\cdot p $. Because $g$ and $f$ are isomorphisms and $F$ induces a linear bijection between hom spaces, $\lambda$ is non zero and we can replace $g$ by $\lambda\cdot g$. This concludes the proof.
\end{proof}

\section{A family of Fractionally Calabi--Yau Posets}
\label{section2}
The goal of this section is to prove Theorem \ref{thm1}. We use the combinatorial family of antichain modules introduced in \cite{Yildirim_2018} and show that it satisfies the conditions of Theorem \ref{claim:Gen}. For the convenience of the reader and because it will be used heavily in the rest of the article, we recall Y\i ld\i r\i m's combinatorial framework and proof, modifying the statements when needed to obtain results in the derived category.
\subsection{Grids and their order ideals}%%%%%%%%%%%%%%%%%%%%%%%%%%%%%%%%%%%%%%%%%
\label{subsec:grids}
Let $m$ and $n$ be positive integers and $G_{m, n}$ be the product of two total orders of size $m$ and $n$ respectively. Recall that an order ideal $I$ of a poset $P$ is a subset $I\subseteq P$ which is downward closed, \emph{i.e.} \[\text{if }x\in I \text{ and } y\leq x \text{ then }y\in I.\] Order ideals of a poset can be ordered by inclusion and form a distributive lattice when equipped with the union and the intersection of subsets. We denote by $J_{m,n}$ the lattice of order ideals of $G_{m,n}$. %%%%%%%%%%    
%Proposition \ref{prop:injections} gives a discription of $J_{m,n}$ in terms of configurations. It plays crucial role in the proof of Theorem \ref{thm1}. Visualising partitions as paths in grids or abaci as in Example \ref{ex:confToAbac} will be more helpful in the Section \ref{section3}. 

\begin{wrapfigure}[10]{r}{0.15\textwidth}\centering
	\vspace{-1.5em}
	\begin{tikzpicture}
	\draw[step = 0.5cm, color = gray](0,0) grid (2,3);
	\draw[thick](.5, 0)--(.5, .5)--(1, .5)--(1, 1)--(1.5, 1)--(1.5, 3)--(2,3);
	\filldraw (.5,.5) circle (1pt);
	\filldraw (1,1) circle (1pt);
	\filldraw (1.5,3) circle (1pt);
	\filldraw (2,3) circle (1pt);
	\end{tikzpicture}
	\caption{Order ideal of the $5\times 7$ grid.}\label{fig:grid}
\end{wrapfigure}
\par This lattice possesses several interesting description. 
Take an order ideal or down closed set $I$ of $G_{m,n}$. It is completely determined by a sequence  of length $m$ of non negative integers by counting and listing elements of $I$ contained in each chain of the grid from $m$ to $1$. Figure \ref{fig:grid} depicts an order ideal of the $5\times 7$ grid as a \defn{path}. The elements of the order ideals of Figure \ref{fig:grid} should be read as the grid points that lie below the drawn path. Counting from left to right the number of elements of the order ideal contained in each chain, or column, we get the non zero integer sequence $(0, 2, 3, 7, 7)$ which is non decreasing. We obtain such a non decreasing sequence when applying this procedure for any order ideal $I$ precisely because order ideals are down closed. Paths in grids like in Figure \ref{fig:grid} will be used throughout the current section and Subsection \ref{subsec:homsets1} to depict order ideals.
%One can also think of partitions as paths in an $m\times n$ grid as depicted in the figure for the partition $(0, 2, 3, 7, 7)$. For a non decreasing sequence $(a_1, \dots, a_m)$ the path is obtained by putting a dot at height $a_i$ in the $i^{th}$ column from the left and then take the minimal path going through this dots. If $a_i = 0$ put no dot.
%Because $I$ is down closed, the amount of elements in the chain does not decrease.
%counting the number of elements that belong to $I$ in each column, with increasing first value, gives a monotone sequence which completely determines the ideal. 
We also just described a bijection
\begin{equation}
J_{m,n} \cong \{(a_1, \dots, a_m)| a_i\in \llbracket 0, n\rrbracket \text{ and } a_1\leq \dots, \leq a_n\}
\end{equation} with non decreasing sequences.
If the second set is equipped with termwise comparison this is an isomorphism of posets. We call these non decreasing sequences \defn{partitions}. They can also be written as 
\[\alpha = (\lambda_1^{\mu_1}, \dots, \lambda_r^{\mu_r} )\]
with $\sum_i \mu_i = m$, where $\mu_i$ encodes the multiplicity of the value $\lambda_i$ and $\lambda_i \not = \lambda_j$ if $i\not = j$.
These partitions can classically be counted as follows: choosing a partition amounts to putting $m$ balls corresponding to coefficients $(a_1, \dots, a_m)$ into $n+1$ boxes, the possible values of the coefficients. Which also amounts to placing $n$ sticks in $m+n$ possible slots. This means there are exactly $${m+n\choose m}$$ partitions. 
Next we introduce two maps that send these non decreasing sequences of length $m$ to increasing sequences of length $m$. Let \[\mathcal{Z} = \{-m, \dots, -1, 0, 1, \dots, n\}\] be a set of representatives of $\mathbb{Z}/(m+n)\mathbb{Z}$. A \defn{configuration} $C$ of $\mathcal{Z}$ is a subset of size $m$ of $\mathcal{Z}$. We write it $C = \{c_1< \dots < c_m\}$ is an increasing sequence naively using the order relation on $\mathbb{Z}$.  We write $C_{m,n}$ the set of configurations of length $m$ on $\mathcal{Z}$. Choosing a configuration amounts to picking $m$ distinct elements from a set of cardinal $m+n+1$ so the cardinality of $C_{m,n}$ is \[{m+ n + 1\choose m}.\]
Given a partition $\alpha$ we can construct a configuration containing $\alpha$'s coefficients in its non negative side and encoding the multiplicities of $\alpha$ in its negative side. 
Write 
\[x_i = \sum_{k = 1}^i \mu_i\]
to record the index of the last occurence of the $i^{th}$ coefficient. It will be called the \defn{ending index}. We set $x_0 = 0$ as a convention. The index $x_{i-1} + 1$ is the first occurence of $\lambda_i$ and will be called the \defn{starting index}. Think of the negative side as the indices of the elements of the sequence $\alpha$ but with a minus sign. Out of the many available options to encode the multiplicities, here are two that turn out to fit the problem perfectly:
\begin{itemize}[leftmargin=*, topsep=0pt, partopsep=0pt, itemsep = 0pt]
\item keep all negative elements except the opposite of the starting index of each coefficient. The resulting configuration is called the \defn{left configuration} associated to $\alpha$, and we denote by it $L_{\alpha}$. The map sending $\alpha$ to $L_{\alpha}$ is denoted by $\phi_l$;
\item keep all negative elements except the opposite of the ending index of each coefficient. The resulting configuration is called the \defn{right configuration} associated to $\alpha$, and we denote it by $R_{\alpha}$. The map sending $\alpha$ to $R_{\alpha}$ is denoted $\phi_r$.
\end{itemize}
\begin{ex}\label{ex:phiRphiL}
Take $n = 7$, $m = 5$ and consider the partition $a = (0, 2, 3, 7, 7)$. So we have $r = 4$ and 
$$x_1 = 1, x_2 = 2, x_3 = 3 \text{ and } x_4 = 5.$$
The associated left and right configuration are respectively
$$\{-5< 0<2<3<7\} \text{ and } \{-4<0<2<3<7\}.$$
\end{ex}
\begin{prop}[\protect{\cite[Proposition~3.3]{Yildirim_2018}}]\label{prop:injections}
The maps $\phi_l$ and $\phi_r$ are injective.
\end{prop}
\begin{proof}Partitions are entirely determined by their coefficients and multiplicities. These can be recovered from the positive elements of a configurations and its negative gaps, \emph{i.e.} missing values, respectively. 
\end{proof}They are not surjective as $\phi_l(\alpha)$ cannot contain $-1$ and $\phi_r(\alpha)$ cannot contain $-m$.
To visualise configurations consider a table with $m+n+1$ columns and put a dot, or a bead in the columns corresponding to the elements of the configuration at hand. We call this the \defn{abacus} associated to the configuration.
\begin{ex}\label{ex:confToAbac} The abacus associated with the right configuration of the partition $a$ from Example \ref{ex:phiRphiL} is as follows
\begin{center}\begin{tabular}{c c c c c | c c c c c c c c}
-5&-4&-3&-2&-1&0&1&2&3&4&5&6&7 \\
\hline
&$\bullet$&&&&$\bullet$&&$\bullet$&$\bullet$&&&&$\bullet$.
\end{tabular}
\end{center}
Note that we used the map $\phi_r$ for this configuration.
\end{ex}
Each bead of the \defn{negative side} adds one to the multiplicity of one of the coefficients. If we are using the map $\phi_r$ to obtain the configuration, the bead sitting in the $-k^{th}$ column, is associated to the $k^{th}$ bead to its right. Indeed, gaps in the negative side indicate changes of coefficients. If the columns $-1$ to $-k+1$ are empty then exactly $(k-1)^{th}$ coefficients end before the $k^{th}$ element of the sequence and the claim holds. Any extra bead between $-1$ and $-k+1$ removes one such gap. If there are $s$ such beads in the negative side, the coefficient  we are looking for corresponds to the $(k-s)^{th}$ bead in the positive side. Note that $s$ is at most $k-1$, so the procedure always yields a bead in the positive side.

\begin{ex}
\begin{figure}[h!]
\includegraphics[width=9.6cm]{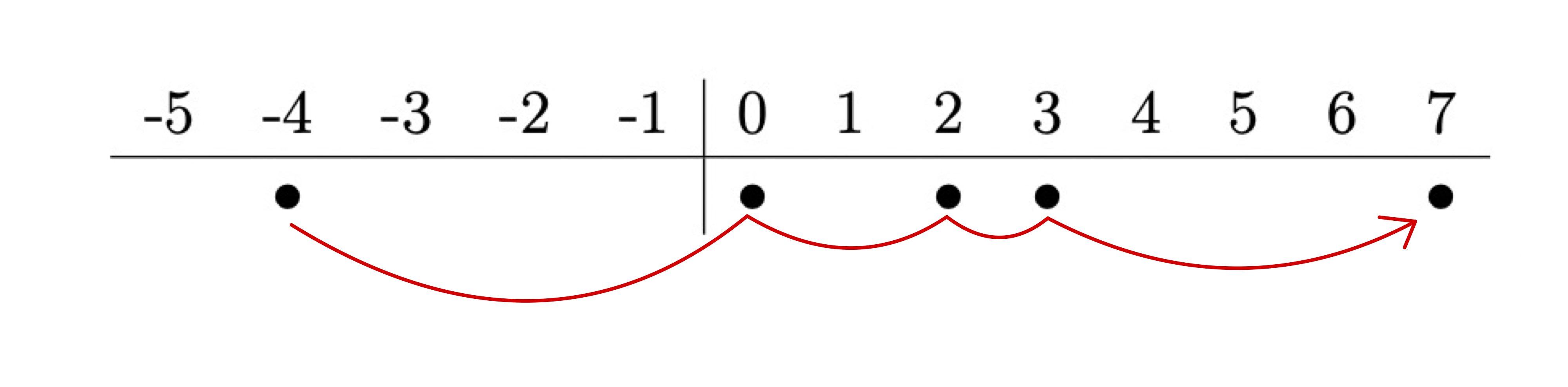}
\centering
\caption{Illustrating the partial inverse of $\phi_r$}
\label{fig:exPhirInverse}
\end{figure}
Figure \ref{fig:exPhirInverse} illustrates this idea on the abacus of Example \ref{ex:confToAbac}. Its only negative bead is in column $-4$. Counting four beads to the right of it, we find the bead in column $7$. This step is illustrated by the long red arrow. The bead in column $7$ corresponds to the coefficient $\lambda_4$. It is the only coefficient with multiplicity 2 in the partition $a$.
\end{ex}

Similarly, if the map $\phi_l$ was used, negative beads can be associated to their positive counterparts by skipping $k-1$ beads to the left. 
%%%%%%%%%%%%%%%%%%%%%%%%%%%%%%%%%%%%%%%%%%%%%%%%%%%%%%%%%%%%%%
\subsection{A Family of antichain modules} 
To apply the machinery presented in Section \ref{section4}, we identify a well behaved family of antichains. To create an antichain below a certain element $\alpha$ it suffices to identify transformations on $\alpha$ whose respective support cannot be compared. The family we discuss was discovered by Y\i ld\i r\i m and used to prove that the Coxeter transformation of $J_{m,n}$ is periodic. The rest of the paper will focus on these antichains and their associated modules. We introduce some extra combinatorial data before discussing the antichains themselves, their associated modules and projective resolutions.

\paragraph{Enhancements} 
We consider several antichains below $x$ for each $x\in J_{m,n}$. These different antichains can be encoded as decorations or \defn{enhancements} of $x$. Recall that Theorem \ref{claim:Gen} only requires one antichain below each element of the lattice. Y\i ld\i r\i m introduced this apparent excess of antichain in order to construct a family of antichain modules that is stable under the Serre functor and thus compute its orbits. This is recalled here in Proposition \ref{prop:ProjToInjIso}.

\begin{Def}\label{def:RHP}
A \defn{right enhanced partition} is a sequence $$(\lambda_1^{\mu_1}, \dots, \lambda_r^{\mu_r}| n^{\mu_{r+1}} )$$ 
where multiplicities $\mu_1, \dots, \mu_{r+1}$ sum to $m$. Note that we allow $\mu_{r+1} = 0$. If  $\mu_{r+1} \neq 0$ we say the partition is strictly enhanced. We denote by $E^R_{m,n}$ the set of right enhanced partitions. Similarly a \defn{left enhanced partition} is a sequence 
$$(0^{\mu_0} | \lambda_1^{\mu_1}, \dots, \lambda_r^{\mu_r}).$$
\end{Def}

\begin{wrapfigure}[9]{r}{.5\textwidth}\centering
		\begin{tikzpicture}[scale = 1.5]
			\draw[step = 0.5cm, color = gray](0,0) grid (3.5,2);
			\draw[thick](.5, 0) -- (1, 0) -- (1, 1) -- (2, 1) -- (2, 1.5) -- (2.5, 1.5) -- (2.5, 2) --(3, 2)--(3.5, 2);
			\draw[thick, red](3, 2) -- (3.5, 2);
			\foreach \n in {(0.5, 0), (1, 1), (1.5, 1), (2, 1.5) , (2.5, 2)}{
			\filldraw \n circle (1pt);}
			\filldraw[red] (3, 2) circle (1pt);
			\filldraw[red] (3.5, 2) circle (1pt);
			\draw[thick](2.9, 1.7) -- (2.9, 2.1);
		\end{tikzpicture}
	\caption{Enhanced partitions in grids}\label{fig:enhancements}
\end{wrapfigure}
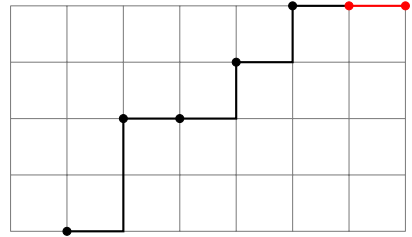

\begin{ex}\label{ex:enhancements}The right enhanced partition $\beta = (0, 1, 3^2, 4, 5 | 5^2)$ can be represented as the path on the right with the values behind the bar coloured in red. See Figure \ref{fig:enhancements}.

\end{ex}
The set of left enhanced partitions will be written $E^L_{m,n}$. A partition with $\mu_{r+1} = 0 = \mu_0$ is called \defn{plain}. Just like for plain partitions, we count $${m+n+1\choose m}$$ right enhanced partitions as well as left enhanced ones.
The map $\phi_r$ naturally be extends to right enhanced partitions without any change, making it a bijection. By "without any change" we mean that, just as for plain partitions, we remove the values $-x_1$ to $-x_r$ in the negative side and do not remove $-x_{r+1} = -m$. Similarly, the map $\phi_l$ extends to left enhanced partitions, not removing $-1$ from the negative side when $\mu_0 > 0$.
\begin{ex}
Consider enhanced versions of the partition $a$ from Example \ref{ex:phiRphiL}. The left enhanced partitions associated to $(0, 2, 3, 7\,|\,7)$ and $(0, 2, 3 \,|\, 7, 7)$ are 
$$\{-5<0<2<3<7\} \text{ and } \{-5<-4<0<2<3\}$$
respectively. The right configuration associated to the left enhanced partition $(0\,|\,2,3, 7, 7)$ is 
$$\{-5<-1<2<3<7\}.$$ These configurations are represented in an abacus as follows:
\begin{center}\begin{tabular}{c c c c c | c c c c c c c c}
-5&-4&-3&-2&-1&0&1&2&3&4&5&6&7 \\
\hline
$\bullet$&&&&&$\bullet$&&$\bullet$&$\bullet$&&&&$\bullet$\\
$\bullet$&$\bullet$&&&&$\bullet$&&$\bullet$&$\bullet$&&&&\\
$\bullet$&&&&$\bullet$&&&$\bullet$&$\bullet$&&&&$\bullet$\\
\end{tabular}
\end{center}
\end{ex}

\paragraph{Corresponding antichains}
Our antichains are obtained by modifying the coefficients and leaving multiplicities unchanged. For a right enhanced partition $\alpha = (\lambda_1^{\mu_1}, \dots, \lambda_r^{\mu_r}|n^{\mu_{r+1}})$ define the \defn{mutable coefficients} to be $S_{\alpha} = \{\epsilon, \dots, r\}$ the indices corresponding to non zero coefficients. The number $\epsilon$ is either $1$ or $2$. Please remark that this excludes the coefficients beyond the enhancement bar.

\begin{Def}\label{def:qj}
Let $\alpha = (\lambda_1^{\mu_1}, \dots ,  \lambda_r^{\mu_r}|n^{\mu_{r+1}})$ be a right enhanced partition. For any  subset $J$ of $S_{\alpha}$ define a new partition $q_J(\alpha) = ((\lambda_1')^{\mu_1}, \dots, (\lambda'_r)^{\mu_r}|n^{\mu_{r+1}})$ by 
$$\lambda'_i = \begin{cases}
		\lambda_i - 1 &\text{ if } i\in J,\\
		\lambda_i &\text{ otherwise.}
		\end{cases}
$$
\end{Def}
Consider now the set
	\begin{equation}\label{def:YildirimModules}
		C_{\alpha} = \{q_i(\alpha)| i\in S_{\alpha}\}.
	\end{equation}
Because $q_i(\alpha)$ and $q_j(\alpha)$ differ from $\alpha$ at different indices, their associated plain partitions form an antichain and we denote $\PRes_{\alpha}$ the perfect complex associated with it. 
	\begin{wrapfigure}[9]{r}{.4\textwidth}\centering
		\begin{tikzpicture}[scale = 1.5]
			\draw[step = 0.5cm, color = gray](0,0) grid (3.5,2);
			\draw[thick](0.5, 0) -- (1, 0) -- (1, .5) -- (2, .5) -- (2, 1.5) -- (3, 1.5) -- (3, 2);
			\draw[thick, dotted](1, .5) -- (1, 1) -- (2, 1);
			\draw[thick, dotted](2.5, 1.5) -- (2.5, 2) -- (3, 2);
			\draw[thick] (3, 2) -- (3.5, 2);
			\draw[red, ->] (1.5, .9) -- (1.5, .6);
			\draw[red, ->] (2.75, 1.9) -- (2.75, 1.6);
			\filldraw (.5, 0) circle (1pt);
			\filldraw (1, .5) circle (1pt);
			\filldraw (1.5, .5) circle (1pt);
			\filldraw (2, 1.5) circle (1pt);
			\filldraw (2.5, 1.5) circle (1pt);
			\filldraw (3, 2) circle (1pt);
			\filldraw (3.5, 2) circle (1pt);
			\draw[thick](2.9, 1.7) -- (2.9, 2.1);
		\end{tikzpicture}
		\caption{Illustration of $q_K$}\label{fig:antichains}
	\end{wrapfigure}
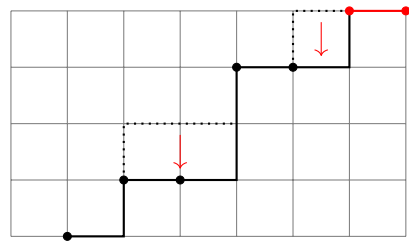
We have opted to define $q_i(\alpha)$ as enhanced partitions as these transformations will also parametrised extensions between objects $\PRes_{\alpha}$ described in the next section. However the antichains associated with this construction are made up of elements of the poset, which are plain partitions.
\begin{ex}\label{ex:antichains}Consider the right enhanced partition $\beta$ from Example \ref{ex:enhancements}. Then we have $S_{\beta} = \{2, 3, 4, 5\}$. Picking the subset $K = \{3, 5\}$ of $S_{\beta}$ yields $q_K(\beta) = (0, 1, 2^2, 4^2| 5^2)$. See Figure \ref{fig:antichains}. 
\end{ex}
\begin{prop}\label{prop:YildBool} The antichain $C_{\alpha}$ is a boolean antichain below $\alpha$. 
\end{prop}
\begin{proof}First we check that the antichain is \defn{strong}. Take $I, J \subseteq S_{\alpha}$ such that $|I| = |J| > 0$ and $I\not = J$. Take $i\in I\setminus J$ and $j\in J\setminus I$. Then it holds that
\[\lambda_i - 1 = q_I(\alpha)_{x_i} < q_J(\alpha)_{x_i} = \lambda_i\]
and symmetrically at the index $x_j$. Hence $q_I(\alpha)$ and $q_J(\alpha)$ cannot be compared and the antichain is strong. Now we check that the antichain is \defn{intersective}. Recall that the join of two partitions is the termwise maximum. We take $I, J\subset S_{\alpha}$ to be non empty. We write 
\[\gamma = q_I(\alpha)\vee q_J(\alpha).\]
For each $i\in S_{\alpha}$, and every $k\in ]x_{i-1}, x_i]$ we have 
\[\gamma_k = \begin{cases}
\lambda_i - 1 &\text{if } k\in I\cap J\\
\lambda_i &\text{otherwise.}
\end{cases}\]
Hence $\gamma = q_{I\cap J}(\alpha).$
\end{proof}

\paragraph{Associated Intervals}
\begin{prop}[\protect{\cite[Proposition~2.13]{Yildirim_2018}}]\label{prop:f}Let $\alpha$ be a right enhanced partition. Then $\PRes_{\alpha}$ is a projective resolution of the interval $[f(\alpha), \alpha]$ where the function $f$ is defined by
\begin{equation}\label{eq:ProjInt}
f : (\lambda_1^{\mu_1}, \dots, \lambda_r^{\mu_r} | n^{\mu_{r+1}}) \mapsto (0^{\mu_1-1}|\lambda_1^{\mu_2}, \dots, \lambda_r^{\mu_{r+1}+1}).
\end{equation}

\end{prop}
It is more convenient to define the image of a right enhanced partition by the function $f$ to be a left enhanced partition. However, the interval $[f(\alpha), \alpha]$ is an interval of $J_{m,n}$, \emph{i.e.} is the interval with bounds the corresponding plain partitions.
\begin{proof}An element $\beta$ of the lattice is in the support of the antichain module associated
to $\PRes_{\alpha}$ if and only if $\beta\leq \alpha$ and for all $i\in S_{\alpha}$, we have $\beta \not\leq q_i(\alpha)$. Because the partition $q_i(\alpha)$ differs from $\alpha$ only between the indices $x_{i-1} + 1 $ and $x_i$, there exists $k \in [x_{i-1}+1, x_i]$ such that $b_k = a_k = \lambda_i$. The sequence $\beta$ is increasing so this is equivalent to $b_{x_i}= \lambda_i$ for all $i\in S_{\alpha}$. The partition $f(\alpha)$ satisfies these conditions. For any other $\beta$ satisfying them, for $k \in [x_{i-1}+1, x_i]$, we have $f(\alpha)_k =\lambda_{i-1} = b_{x_{i-1}}\leq b_k$, hence $\beta\in [f(\alpha), \alpha]$.
\end{proof}

In this proof we used the fact that the support of $\PRes_{\alpha}$ is the set
\begin{equation}\label{eq:SuPalpha}
\{\beta \leq \alpha | \forall i\in S_{\alpha}, \beta_{x_{i}} = \lambda_i\}.
\end{equation}
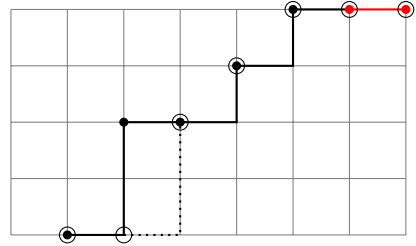
\begin{wrapfigure}[8]{r}{.4\textwidth}\centering
	\begin{tikzpicture}[scale =1.5]
		\draw[step = 0.5cm, color = gray](0,0) grid (3.5,2);
		\draw[thick](.5, 0) -- (1, 0) -- (1, 1) -- (2, 1) -- (2, 1.5) -- (2.5, 1.5) -- (2.5, 2) --(3, 2);
		\draw[thick](3, 2) -- (3.5, 2);
		\foreach \n in {(0.5, 0), (1, 1), (1.5, 1), (2, 1.5) , (2.5, 2)}{
		\filldraw \n circle (1pt);}
		\filldraw (3, 2) circle (1pt);
		\filldraw (3.5, 2) circle (1pt);
		\draw[thick, dotted](1, 0) -- (1.5, 0) -- (1.5, 1);% -- (2, 1.5) -- (2.5, 1.5) -- (2.5, 2) --(3, 2);
		\foreach \n in {(0.5, 0), (1, 0), (1.5, 1), (2, 1.5) , (2.5, 2), (3, 2), (3.5, 2)}{
		\draw \n circle (2pt);}
		\draw[thick](2.9, 1.7) -- (2.9, 2.1);
	\end{tikzpicture}
	\caption{partitions $\beta$ and $f(\beta)$}\label{fig:pathFalpha}
\end{wrapfigure}
This highlights the role of the indices $x_i$ for $i\in S_{\alpha}$. These indices will serve as comparison points between our partitions in many proofs to follow.

\begin{ex}
Consider the partition $\beta$ from Example \ref{ex:enhancements}. Then $f(\beta) = (|0, 1^2, 3^1, 4, 5^3)$ with the corresponding path on the right in Figure \ref{fig:pathFalpha}. Please note how the value of $\beta$ and $f(\beta)$ match at the ending indices $1, 3, 4$ and $5$ and how the values of $f(\beta)$ are minimal given those constraints.
\end{ex}

%%%%%%%%%%%%%%%%%%%%%%%%%%%%%%%%%%%%%
\subsection{Y\i ld\i r\i m's theorem}\label{subsec:YildThm}
The following key result is a categorified version of \cite[Proposition~4.2]{Yildirim_2018}.
\begin{prop}\label{prop:YildirimThm}Let $\alpha$ be a right enhanced partition. Then 
	$$\Serre^{m+n+1}(\PRes_{\alpha}) \cong \PRes_{\alpha}[mn].$$
\end{prop}

First we describe the action of the Serre functor on the object $\PRes_{\alpha}$. Recall that it sends the projective indecomposable $P_{\alpha}$ to the injective indecomposable $I_{\alpha}$.

\begin{Def} We call $\mathcal{I}_{\alpha}$ the image of the projective resolution $\PRes_{\alpha}$ under the Serre functor. Because $\PRes_{\alpha}$ is a complex of projective modules, we get $\mathcal{I}_{\alpha}$ by tensoring the components of  $P_{\alpha}$ by $\IncAlg^{*}$ which gives us
	\begin{equation}\label{eq:InjRes}
	\mathcal{I}_{\alpha} : 0\to I_{q_{S_{\alpha}}(\alpha)}\xrightarrow{\partial_r}\bigoplus_{\substack{J\subseteq S_{\alpha},\\|J| = r-1}} I_{q_J(\alpha)}	
						\to\dots\to\bigoplus_{\substack{J\subseteq S_{\alpha},\\ |J| = r-k}} I_{q_J(\alpha)} \overset{\partial_{r-k}}{\to}\dots\to I_{\alpha}\to 0.
\end{equation}									
\end{Def}

For any right enhanced partition $\alpha$, the complex $\mathcal{I}_{\alpha}$ is an injective resolution for some module. Its homology is concentrated in degree $|S_{\alpha}|$. To compute the homology, define the map
	\begin{equation}\label{eq:DefG}
	g : (0^{\alpha_0}| \lambda_1^{\alpha_1}, \dots, \lambda_r^{\alpha_r}) \mapsto (\lambda_1^{\alpha_0+1}, 					\lambda_2^{\alpha_1}, \dots, \lambda_r^{\alpha_{r-1}}|n^{\alpha_r - 1})
	\end{equation}
which is a counterpart to the map $f$, and set the following enhancement on the partition $q_{S_{\alpha}}(\alpha)$
	\begin{equation}\label{eq:deltaEnhancement}
	\begin{cases}
	(0^{\alpha_1}|(\lambda_2 - 1)^{\alpha_2}, \dots, (\lambda_r-1)^{\alpha_r}, n^{\alpha_{r+1}}) & \text{if } \lambda_1 = 0\\
	(|(\lambda_1 - 1)^{\alpha_1}, \dots, (\lambda_r-1)^{\alpha_r}, n^{\alpha_{r+1}}) & \text{otherwise.}
	\end{cases}
	\end{equation}
We denote $\delta$ the map sending $\alpha$ to this enhancement of $q_{S_{\alpha}}(\alpha)$. Then the support of the homology of $\mathcal{I}_{\alpha}$ in degree $|S_{\alpha}|$ is the interval
\begin{equation}\label{eq:IntIm}
[\delta(\alpha), g\circ \delta(\alpha)].
\end{equation}
The map $\delta$ amounts to representing the corresponding configuration in an abacus using $\phi_r$, shift all its beads one step to the left and interpret the configuration using $\phi_l$. The maps $f$ and $g$ are further related by the following lemma. 

\begin{lem}\label{lem:fginverse}
The functions $f$ and $g$ are inverse of one another.
\end{lem}

\begin{proof}
To see that it suffices to compute both compositions

	\begin{align*}
	g\circ f : (\lambda_1^{\alpha_1}, \dots, \lambda_r^{\alpha_r} | n^{\alpha_{r+1}}) &\mapsto (0^{\alpha_1-1}|, 		\lambda_1^{\alpha_2}, \dots, \lambda_{r-1}^{\alpha_r}, \lambda_r^{\alpha_{r+1}+1})\\
	&\mapsto (\lambda^{\alpha_1}, \lambda_2^{\alpha_2}, \dots, \lambda_r^{\alpha_r}|n^{\alpha_{r+1}+1 - 1})
	\end{align*}

and, almost dually we have : 

	\begin{align*}
	f\circ g : (0^{\alpha_0}| \lambda_1, \dots, \lambda_r^{\alpha_r}) &\mapsto (\lambda_1^{\alpha_0 +1}, 			\lambda_2^{\alpha_1}, \dots, \lambda_r^{\alpha_{r-1}} | n^{\alpha_r - 1})\\
	&\mapsto (0^{\alpha_0 +1 -1}| \lambda_1, \dots, \lambda_{r-1}^{\alpha_{r-1}}, \lambda_r^{\alpha_r -1 + 1}).
	\end{align*}
\end{proof}
The key for Y\i ld\i r\i m's theorem is the following proposition which mirrors the combinatorics of Proposition 4.2 in \cite{Yildirim_2018} but with slightly different combinatorial objects.
\begin{prop} \label{fact:fShift}\label{fact:gShift}
On the abacus, the map $f$ is a shift to the left of the negative beads or is simply $f(\alpha) = \phi^{-1}_l\circ \phi_r(\alpha)$. Conversely $g$ is the reinterpretation of the configuration using $\phi_r$ instead of $\phi_l$ \emph{i.e.} $g(\alpha) = \phi^{-1}_r\circ \phi_l(\alpha)$.
\end{prop}
\begin{proof}Both $\alpha$ and $f(\alpha)$ have the same \emph{plain} coefficients $\lambda_1, \dots, \lambda_r$ of $\alpha$ hence their respective abaci have the same positive beads regardless of the map used. To see that they coincide in the negative side notice that 
\[x^{f(\alpha)}_{i-1}+1 = x_{i}^{\alpha}\]
for $i\in S_{\alpha}$. The result on $g$ follows from Lemma \ref{lem:fginverse}.
\end{proof}
More importantly we can now show that Y\i ld\i r\i m's family of intervals is stable under the Serre functor.
Write $\tilde{f} = g\circ \delta$.
\begin{prop}\label{prop:ProjToInjIso} 
Let $\alpha$ be a left enhanced partition. Then $\Serre(\mathcal{P}_{\alpha}) \simeq \mathcal{P}_{\tilde{f}(\alpha)}[|S_{\alpha}|]$.
\end{prop}
\begin{proof} Recall that $\Serre(\mathcal{P}_{\alpha})$ is isomorphic to the interval $[ q_{S_{\alpha}}(\alpha), g(q_{S_{\alpha}}(\alpha))]$ shifted by $|S_{\alpha}|$. But because $f$ and $g$ are inverse of one another this is the interval 
	$$[f(\tilde{f}(\alpha)), \tilde{f}(\alpha)] $$
which itself is isomorphic to $\mathcal{P}_{\tilde{f}(\alpha)}$ by equation (\ref{eq:ProjInt}). This gives us the result.
\end{proof}
Hence describing the action of the Serre functor on the Y\i ld\i r\i m modules amounts to describing the action of $\tilde{f}$ on abaci. This turns out to be quite simple. We recover the same transformation as in the proof of Proposition 4.2 in  \cite{Yildirim_2018} but with different intermediate maps.
\begin{rem} \label{lem:Shift}
The map $\tilde{f}$ acts on right abaci associated to right enhanced partitions as a shift to the left. This is clear form the description in terms of abaci of $g$ in Proposition \ref{fact:gShift} and of $\delta$ below equation (\ref{eq:IntIm}).
\end{rem}
\begin{proof}[Proof of Proposition \ref{prop:YildirimThm}]Remark \ref{lem:Shift} implies that $\tilde{f}^{n+m+1}$ is the identity on configurations hence on partitions meaning that $S^{m+n+1}(\mathcal{P}_{\alpha})$ is isomorphic to $\mathcal{P}_{\alpha}$ with a certain shift. To compute that shift, recall that applying the Serre functor to $\mathcal{P}_{\alpha}$ shifts the resolution by $|S_{\alpha}|$ to the left so the total shift is: 
$$\sum_{i=1}^{m+n+1} |S_{\tilde{f}^{i}(\alpha)}|.$$
But $|S_{\alpha'}|$ is the number of non zero beads in the positive side of the right abacus associated to $\alpha'$. When applying $\tilde{f}$ to $\alpha$ a total of $n+m+1$ times, each bead will be in the positive side exactly $n$ times. Because there are $m$ beads, this means the shift is $nm$.
\end{proof}

\begin{proof}[Proof of Theorem \ref{thm1}]Let us check that the family 
	$$\{\PRes_{\alpha} | \alpha \text{ enhanced } (m,n) \text{-partition} \}$$
satisfies the three conditions of Theorem \ref{claim:Gen}. First, no matter the enhancement of a partition $\alpha$, the projective cover of the interval $[f(\alpha), \alpha]$ is the indecomposable projective $P_{\alpha}$. Theorem \ref{prop:YildirimThm} gives the second condition. The third condition follows from Proposition \ref{prop:YildBool} which concludes the proof.
\end{proof}
%%%%%%%%%%%%%%%%%%%%%%%%%%%%%%%%%%%%%%%%%%%%%%%%%%%%%%%%%%%%%%
\section{Tilting to higher Auslander algebras of type $A$}
\label{section3}

We now only consider the partitions that satisfy $\mu_{r+1} = 0$. They are the elements of $J_{m,n}$.
\begin{prop}\label{prop:thicc} The category $\Thicc(\{\PRes_{\alpha} | \alpha\in J_{m,n}\})$ is the category of perfect complexes.% of $J_{m,n}$.
\end{prop}
\begin{proof} In the proof of the main result of Section \ref{section4} we constructed by induction each projective module as the end of a sequence of cones starting in  $\mathcal{Y}_{m,n}$. The argument  only relied on the fact that each projective had a quotient belonging to the family $(\PRes_{\alpha})_{\alpha}$. Restricting the family to the modules associated with plain partitions %, \emph{i.e.} those with the bar in $\alpha$ pushed all the way to the right, 
proves the claim. 
\end{proof}
The goal of the rest of the section is to prove Theorem (\ref{thm:DerEqTypeA}). We start by describing morphisms and extensions between the antichain modules $\PRes_{\alpha}$. After discussing relations between the morphisms we construct a tilting complex and compute its endomorphism algebra. Its indecomposable components will be carefully shifted antichains modules $\PRes_{\alpha}$. %Technical lemmas are left out. The interested reader can find the necessary details in \cite[Chapter 5.1]{gottesman2024these}.
\subsection{Describing Homspaces}
\label{subsec:homsets1}
\begin{wrapfigure}[15]{r}{.4\textwidth}\centering
\includegraphics[scale = 0.5]{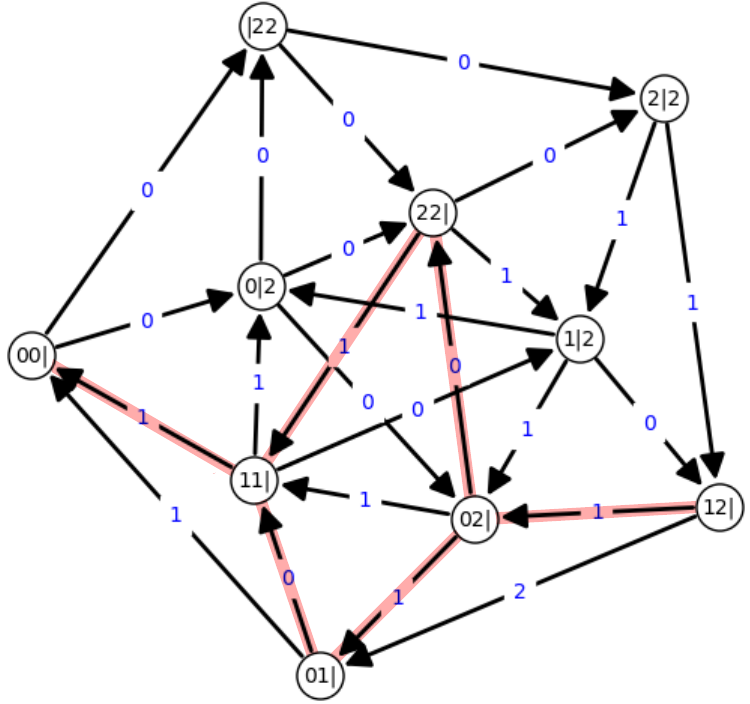}
\caption{Graph of $\mathcal{Y}_{2,2}$ and more}
\label{fig:hom-graph_2_2}
\end{wrapfigure}
We define the category $\mathcal{Y}_{m,n}$ to be the full subcategory of $D^b(J_{m,n})$ whose objects are the objects $\PRes_{\alpha}$ and all their shifts where $\alpha$ is a plain partition. In this subsection we describe morphisms between objects and to identify irreducible morphisms in the category $\mathcal{Y}_{m,n}$. These include morphisms of modules which are concentrated in degree zero as well as morphisms between shifted objects which are in fact extensions. Given a morphism in $\mathcal{Y}_{m,n}$, we will first factor it through an extension of the same degree but which we can easily describe using our antichains. This factorisation yields a degree zero map. We then decompose further these two components. Starting the process of our factorisation with the extension yields a form of asymmetry between morphisms and extensions. It is an artefact of this proof which will be smoothed out in the next subsection.

As a visual source of examples, see Figure \ref{fig:hom-graph_2_2}. It shows morphisms between the objects of the previous subsection. The category $\mathcal{Y}_{2, 2}$ is highlighted. Labels on the arrows indicate the degree in which the morphisms are concentrated. Note that relations do not appear on the figure: many compositions of arrows are in fact zero. For the description of the category spanned by the objects associated to all enhanced partitions see \cite[Sections 5.1 and 5.2]{gottesman2024these}.

Let $\alpha = (\lambda_1^{\mu_1}, \dots, \lambda_r^{\mu_r})$ and $\beta = (l_1^{m_1}, \dots, l_s^{m_s})$ be partitions. In this section we will consistently use the following convention for the ending indices of the coefficients of $\alpha$ and $\beta$ 
\[x_i = \sum_{k\leq i} \mu_k \text{ and } y_j = \sum_{k\leq j} m_k\]
to distinguish data from the two partitions as well as the indices that we might be looking at any given moment. Recall from Theorem \ref{lem:homsBool} that there exists at most one integer $p$ such that 
\begin{equation*}
\dim \Hom (\PRes_{\alpha}, \PRes_{\beta}[p]) \not= 0.
\end{equation*}
If it exists then $\dim \Hom (\PRes_{\alpha}, \PRes_{\beta}[p]) = 1$. This is the case \emph{if and only if} there exists a unique subset $J$ of $S_{\alpha}$ such that $q_J(\alpha) \in [f(\beta), \beta]$. In that case, $p = |J|$. When $p = 0$, we can give different characterisations of this property.

\begin{prop}\label{lem:hom0}The following are equivalent.
	\begin{enumerate}
		\item\label{item:1}There exists a non zero morphism $\phi : \PRes_{\alpha}\to\PRes_{\beta}$.
		\item\label{item:2}The inequalities $f(\alpha)\leq f(\beta)\leq \alpha\leq\beta$ hold.
		\item\label{item:3}The partition $\alpha$ is in $[f(\beta), \beta]$ and for all non empty $J\subseteq S_{\alpha}$ we have \[q_J(\alpha)\not\in [f(\beta),\beta].\]
		\item\label{item:4}The partition $\alpha$ is in $[f(\beta), \beta]$ and $\{\lambda_i |i\in S_{\alpha}\} \subseteq \{l_j|j\in S_{\beta}\}$.
		\item\label{item:5}For all $j\in S_{\beta}$ there exists $i\in S_{\alpha} \cup \{r+1\}$ such that $\lambda_i = l_j$ and $x_{i-1} < y_j \leq x_i < y_{j+1}$ and $\{\lambda_i |i\in S_{\alpha}\} \subseteq \{l_j|j\in S_{\beta}\}$.
	\end{enumerate}
\end{prop}
\begin{proof}
\begin{description}[leftmargin=0cm]
\item[\protect{(\ref{item:1}) $\Leftrightarrow$ (\ref{item:2})}] is the characterisation of morphisms between intervals. See equation (\ref{eq:morphIntS}).
\item[\protect{(\ref{item:1}) $\Leftrightarrow$ (\ref{item:3})}] follows from Theorem \ref{lem:homsBool}.
\item[\protect{(\ref{item:4}) $\Leftrightarrow$ (\ref{item:5})}] Reformulate (\ref{item:4}) using the ideas of the proof of Proposition \ref{prop:f} which we reproduce here. Recall that the interval $[f(\beta), \beta]$ is the set of partitions 
\[\{\gamma | \gamma\leq \beta \text{ and for all }j\in S_{\beta},  \gamma\not\leq q_j(\beta)\}.\]Because the partition $q_j(\beta)$ differs from $\beta$ between indices $y_{j-1}+1$ and $y_j$, this is equivalent to the existence of $k\in [y_{j-1}+1, y_j]$ such that $l_j -1 <a_k = b_k = l_j$. Equivalently, there exists $i\in S_{\alpha}$ such that $a_k = \lambda_i = l_j = b_k$. The equality $a_k = b_k$ is equivalent to saying there is an overlap of the occurrences of the value $l_j = \lambda_i$ in $\beta$ and $\alpha$. This translates into the interlacing  $x_{i-1} < y_j \leq x_i < y_{j+1}$. 
\end{description}
\noindent To complete the proof we argue that (\ref{item:3}) $\Rightarrow$ (\ref{item:4}) $\Rightarrow$ (\ref{item:2}).
\begin{description}[leftmargin=0cm]
\item[\protect{(\ref{item:3}) $\Rightarrow$ (\ref{item:4})}]Assume that for all $i\in S_{\alpha}$, $q_i(\alpha)\not\in [f(\beta), \beta]$. Thus there exists $j\in S_{\beta}$ such that $(q_i(\alpha))_{y_j} < l_j$. Because $\alpha$ and $q_i(\alpha)$ differ only for indices between $x_{i-1}+1$ and $x_i$ we conclude that $\lambda_i = l_j$. The inclusion $\{\lambda_i |i\in S_{\alpha}\} \subseteq \{l_j|j\in S_{\beta}\}$ follows.
\item[\protect{(\ref{item:4}) $\Rightarrow$ (\ref{item:2})}] 
Remember that in $f(\beta)$, the value $l_j$ runs from $y_j$ to $y_{j+1}-1$. Hence, the interlacing condition implies that $(f(\beta))_{x_i} = \lambda_i$ and $f(\beta)\in[f(\alpha), \alpha]$ because for all $i\in S_{\alpha}$, we have $q_i(\alpha)\not\leq f(\beta)$.  
\end{description}
\end{proof}
To describe some degree zero morphisms, we introduce new transformations on $J_{m,n}$.
\begin{Def}\label{def:allowedDual}
Let $\alpha = (\lambda_1^{\mu_1}, \dots, \lambda_r^{\mu_r})$ be an enhanced partition. Let $i$ be in $\{1, \dots, r-1\}$. If $\mu_i > 1$ define the partition $p_i(\alpha) = (\lambda_1^{m_1}, \dots, \lambda_r^{m_r})$ with multiplicities
\begin{equation*}
m_j = \begin{cases}\mu_j - 1 & \text{if } j=i,\\
			\mu_j + 1 &\text{if } j = i+1,\\
			\mu_j &\text{otherwise.}
	\end{cases}
\end{equation*}
There is a special case when $i = 1$, $\mu_1 = 1$ and $\lambda_1 = 0$. We still define the partition $p_1$ in that case by
\begin{equation*}
p_1(\alpha) = p_1(0^1, \lambda_2^{\mu_2}, \dots, \lambda_r^{\mu_r}) = (\lambda_2^{\mu_2+1}, \lambda_3^{\mu_3}, \dots, \lambda_r^{\mu_r}).
\end{equation*}
 The resulting partition has one coefficient less than $\alpha$ which reduces by one the index of the coefficients which remain in the partition. This will lead to technical difficulties in some proofs. In these two cases we say that $p_i$ is \defn{well defined}.
\end{Def}
Where $q_j$ changes the values of the partition, $p_i$ acts on its multiplicities.
\begin{wrapfigure}[8]{r}{0.5\textwidth}\centering
		\begin{tikzpicture}[scale = 1.5]
			\draw[step = 0.5cm, color = gray](0,0) grid (3.5,2);
			\draw[thick, dotted](0, 0)--(0, 0.5)--(0.5, 0.5);
			\draw[thick, dotted](2, .5)--(2, 1.5)--(2.5, 1.5);
			\draw[thick, dotted](3.5, 2)--(3, 2)--(3, 1.5);
			\draw[thick](0.5, 0)--(0.5, 0.5)--(2.5, 0.5)--(2.5, 1)--(2.5, 1.5)--(3.5, 1.5) -- (3.5, 2);
			\draw[red, ->](0.44, 0.25)--(0.06, 0.25);
			\draw[red, ->] (2.44, 1.25)--(2.06, 1.25);
			\draw[red, ->] (3.44, 1.75)--(3.06, 1.75);
			\filldraw (0.5, 0.5) circle (1pt);
			\filldraw (1, .5) circle (1pt);
			\filldraw (1.5, .5) circle (1pt);
			\filldraw (2, .5) circle (1pt);
			\filldraw (2.5, 1.5) circle (1pt);
			\filldraw (3, 1.5) circle (1pt);
			\filldraw (3.5, 2) circle (1pt);
			\filldraw (3.5, 2) circle (1pt);
		\end{tikzpicture}
		\caption{$p_1$, $p_2$ and $p_3$ on $\alpha = (0, 2^4, 4^2, 5)$}
\end{wrapfigure}
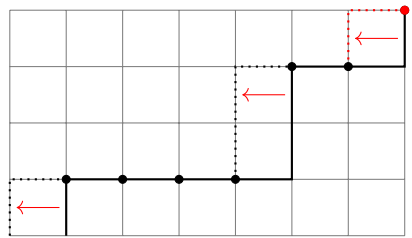
Using Proposition \ref{lem:hom0} and Definition \ref{def:allowedDual} we list some notable degree zero morphisms. 
\begin{cor}\label{nota:Hom0}Consider a right enhanced partition $\alpha$. There exists a non zero morphism $\PRes_{\alpha}\to\PRes_{\beta}$ whenever $\beta= p_i(\alpha)$ and $i\in \{1, \dots, r -1 \}$ such that $p_i$ is well defined.
\end{cor}

\begin{proof}
Notice that if $p_i$ is well defined, then the ending indices of $\alpha$ and $p_i(\alpha)$ are the same except from $x_i^{\alpha} = x_i^{p_i(\alpha)} + 1$. Hence, if $\mu_i > 1$, we have $\alpha_{x_j^{\alpha} -1} = \lambda_i = p_i(\alpha)_{x_j^{p_i(\alpha)}}$ for all $j\in S_{\alpha}$. It follows that $\alpha \in [f(p_i(\alpha)), p_i(\alpha)]$ using the characterisation from equation (\ref{eq:SuPalpha}). We also know from Proposition \ref{prop:f} that $f(p_i(\alpha))$ and $p_i(\alpha)$ have the same values at indices $x_i^{p_i(\alpha)}$ from which we can deduce that $f(p_i(\alpha)) \in [f(\alpha), \alpha]$ using the characterisation equation (\ref{eq:SuPalpha}) again. We conclude with Proposition \ref{lem:hom0}(\ref{item:2}). If $i = 1$ and $\mu_1 = 1$ then none of the non zero ending indices change so the result holds. See Figure \ref{fig:hom0ex} for a visualisation of the intervals. In the figure, the continuous path in black and grey represent the partitions we are considering. They are each associated to a dashed path in the same colour that represents the minimum of their corresponding intervals. The interpolation of the intervals is seen through the succession of decreasing paths in the following order: plain black, plain grey, dashed black and finally dashed grey.
\end{proof}

We now record facts about the combinatorics of the $p_i$ transformations. The notation $J - 1$ refers to the set $\{j-1 |j\in J\}$. Many of the results that follow rely on this sort of manipulations.
\begin{lem}\label{lem:pJFacts}Let $J$ be a subset of $S_{\alpha}$, let $h, j$ be elements of $S_{\alpha}$ and let  $i, k$ be in $\{1, \dots, r\}$. The following assertions hold.
\begin{enumerate} 
\item\label{item:all5} If $p_i$ is well defined  and $i\not = 1$ or $\mu_1 > 1$ then $q_J(p_i(\alpha)) > q_J(\alpha)$ and for any $J' \subseteq S_{\beta}$ different from $J$ but of same cardinal, we have $q_{J'}(p_i(\alpha)) \not > q_J(\alpha)$.
\item\label{item:all6} If $i= 1$, $\mu_1 = 1$ and $\lambda_1 = 0$ then $q_{J-1}(p_i(\alpha)) > q_J(\alpha)$ and for any $J' \subseteq S_{\beta}$ different from $J-1$ but of same cardinal, we have $q_{J'}(p_i(\alpha)) \not > q_J(\alpha)$.
\item\label{item:all7} If both $h$ and $j$ are allowed for $\alpha$ we have $q_hq_j(\alpha) = q_jq_h(\alpha)$.
\item\label{item:all8} If both $p_i$ and $p_k$ are well defined for $\alpha$ we have $p_ip_k(\alpha) = p_kp_i(\alpha)$.
\end{enumerate}
\end{lem}	
\begin{proof}
\begin{enumerate}
\item The transformation $p_i$ is well defined and $i \not = 1$ or $\mu_i > 1$, so $\alpha$ and $p_i(\alpha)$ have the same coefficients $\lambda_1, \dots, \lambda_r$. Hence $q_J(\alpha)$ and $q_J(p_i(\alpha))$ have the same coefficients $\lambda'_1, \dots,  \lambda'_s$. They differ only at the index $x_i$ where $q_J(p_i(\alpha))_{x_i} = \lambda'_{i+1} > \lambda'_{i} = q_J(\alpha)_{x_i}$. Hence $q_J(p_i(\alpha)) \geq q_J(\alpha)$. Now take $J'\subseteq S_{p_i(\alpha)}$ of same cardinal as  $J$ but different from it and let $j$ be in $J'\setminus J$. The partitions $q_{J'}(p_i(\alpha))$ has value $\lambda_j - 1$ between the indices $x'_{j-1} +1$ and $x'_j$, where the values $x'_{j-1} +1$ and $x'_j$ vary depending on whether $i = j-1$, $i = j$ or neither. It is easy to check that in all three cases we have $x'_{j-1}+1 \leq x_j \leq x'_j$. At the same time partition $q_J(\alpha)$ has value $\lambda_j$ at index $x_j$. Hence $q_{J'}(p_i(\alpha)) \not\geq q_J(\alpha)$.
\end{enumerate}
The proofs of the remaining items are left for the reader hoping the one we gave provide enough intuition on how the combinatorics of partitions will be discussed in the rest of the paper.

\end{proof}

\begin{nota}
Let $\alpha$ and $\beta$ be as in Lemma \ref{lem:hom0}. Consider the composition of \defn{canonical} maps
\[P_{\alpha }\xhookrightarrow{\iota_{\alpha}^{\beta}}P_{\beta} \twoheadrightarrow [f(\beta), \beta].\]
By Lemma \ref{lem:hom0}(\ref{item:3}), for $j\in S_{\alpha}$ we have $q_j(\alpha)\not \in [f(\beta), \beta]$. Hence the map above sends generators of $N_C^{\alpha}$ to zero and this factors uniquely through $\PRes_{\alpha} \cong P_{\alpha} / N_{\alpha} \cong [f(\alpha), \alpha]$ providing us with one instance of non zero morphism which we denote $\prescript{0}{}{u_i}$. Because the hom space is one dimensional, any other non zero morphism is proportional to $\prescript{0}{}{u_i^{\alpha}}$.
\end{nota}

Up to homotopy, there exist a unique lift of the canonical morphism $\prescript{0}{}{u_i}$ along the antichain projective resolution $\PRes_{p_i(\alpha)}$.
\begin{lem} Let $i$ be in $\{1, \dots, r\}$ such that $p_i$ is well defined for $\alpha$. Then the lift of $\prescript{0}{}{u_i^{\alpha}}$ along the projective resolutions of $\PRes_{p_i(\alpha)}$ is made of monomorphisms in each degree and all the signs on the maps between indecomposable projective modules are positive.
\end{lem}
\begin{proof}
Denote $\phi_k$ the degree $k$ component of the lift of $\prescript{0}{}{u_i^{\alpha}}$. Let $J$ be a subset of $S_{\alpha}$ of size $k$. If $i\not =1$ or $\mu_1 > 1$ then $J$ can also be seen as a subset of $S_{p_i(\alpha)}$. Applying Lemma \ref{lem:pJFacts}(\ref{item:all5}), the subset $J$ is then the only subset of $S_{p_i(\alpha)}$ satisfying $q_J(\alpha)\leq q_{J}(p_i(\alpha))$. We thus have
\[\phi_k = \oplus_{J\subseteq S_{\alpha}}c_{J} \iota_{q_J(\alpha)}^{q_J(p_i(\alpha))}\]
where $c_J\in \field$. If $i = 1$ and $\mu_1 = 1$, we apply Lemma \ref{lem:pJFacts} (\ref{item:all6}) to get that $J-1$ is indeed the only subset of size $k$ of $S_{p_i(\alpha)}$ for which $q_J(\alpha) \leq q_{J-1}(p_i(\alpha))$. This time we have
\[\phi_k = \oplus_{J\subseteq S_{\alpha}}c_{J} \iota_{q_{J-1}(\alpha)}^{q_J(p_i(\alpha))}.\]
One can then check that $c_J = 1$ makes the $\phi_k$ morphisms commute with the boundary maps. Hence we have an explicit lift of the canonical degree zero map which is a monomorphism in each degree.
\end{proof}

%%%%%%%%
\begin{figure}[h!]
	\begin{minipage}{0.45\textwidth}\centering
		\begin{tikzpicture}[scale = 2]
			\draw[step = .5cm, thin, lightgray] (0, 0) grid (3, 2);
			\draw[thick] (0, .5) -- (1, .5) -- (1, 1) -- (2, 1) -- (2, 2) -- (3, 2);
			\draw[thick, dashed] (0, 0) -- (.5, 0) -- (.5, .5) -- (1.5, .5) -- (1.5, 1) -- (3, 1) -- (3, 2);
			\draw[thick, gray](0, 0.05) -- (.55, 0.05) -- (.55, .45) -- (1.05, .45)-- (1.05, 0.95) -- (2.05, .95) -- (2.05, 1.95) -- (3, 1.95);
			\draw[thick, gray, dashed](0, 0.05) -- (.55, .05) -- (.55, .45) -- (1.55, .45) -- (1.55, .95) -- (2.95, .95) -- (2.95, 2);
			%\draw[thick](3.02, 1.7)--(3.02, 2.1);
			%\draw[thick, gray] (3.05, 1.7)--(3.05, 2.1);
			\draw[->] (0.1, 0.25)--(0.4, 0.25);
		\end{tikzpicture}
		\caption*{$\prescript{0}{}{u_0^{\alpha}}$}\label{fig:hom0ex1}
	\end{minipage}
	\begin{minipage}{0.45\textwidth}\centering
		\begin{tikzpicture}[scale = 2]
			\draw[step = .5cm, thin, lightgray] (0, 0) grid (3, 2);
			\draw[thick] (0, .5) -- (1, .5) -- (1, 1) -- (2, 1) -- (2, 2) -- (3, 2);
			\draw[thick, dashed] (0, 0) -- (.5, 0) -- (.5, .5) -- (1.5, .5) -- (1.5, 1) -- (3, 1) -- (3, 2);
			\draw[thick, gray](0, 0.45) -- (1.55, .45)-- (1.55, 0.95) -- (2.05, .95) -- (2.05, 1.95) -- (3, 1.95);
			\draw[thick, gray, dashed](0, 0.05) -- (1.05, .05) -- (1.05, .45) -- (1.55, .45) -- (1.55, .95) -- (2.95, .95) -- (2.95, 2);
			%\draw[thick](3.02, 1.7)--(3.02, 2.1);
			%\draw[thick, gray] (3.05, 1.7)--(3.05, 2.1);
			\draw[->] (1.1, 0.75)--(1.4, 0.75);
		\end{tikzpicture}
		\caption*{$\prescript{0}{}{u_i^{\alpha}}$}\label{fig:hom0ex2}
	\end{minipage}
	\caption{Illustration of the interpolation condition for certain degree zero morphisms.}\label{fig:hom0}
	\label{fig:hom0ex}
\end{figure}
We will later show that any degree zero morphism decomposes into a composition of $\prescript{0}{}{u_i}$ morphisms.
%We will later see that degree zero morphisms occur from $\PRes_{\alpha}$ to $\PRes_{\beta}$ if and only if there exists a sequence $I = (i_1, \dots, i_k)$ of  $i\in \{1, \dots, r\}$ satisfying  $\beta = p_{i_k}\circ\dots\circ p_{i_1}(\alpha)$ such that the intermediate transformations are well defined. In that case we write $\beta = p_I(\alpha)$. Moreover, the $\prescript{0}{}{u_i}$ will be shown to correspond to irreducible morphisms. Because of the combinatorial nature of the argument we will deal with this later and focus first on the extensions. 
%Similarly, we can identify special extension with canonical realisation and explicit lifts along the projective resolutions.
However we first describe special extensions between $\PRes_{\alpha}$ objects using well chosen subsets $J$ of $S_{\alpha}$.
	\begin{Def}\label{def:allowedYild}
A subset $J$ of $S_{\alpha}$ is \defn{allowed} when for all $i\in J$ such that $i-1\in S_{\alpha}\setminus J$, we have $\lambda_{i-1}< \lambda_{i}-1$.
	\end{Def}
	\begin{ex}Consider the partition $\alpha = (0, 1^2, 3, 4^2, 5^2)$ with $S_{\alpha} = \{2, 3, 4, 5\}$ and $\mu_{r+1} = 0$. Then the subset $J = \{2, 3, 4\}$ and $\{2\}$ are allowed while $\{4\}$ and $\{5\}$ are not.
	\end{ex}
We leave the proof of following Lemma to the reader. It records some useful facts about allowed subsets and the combinatorics of the $p_i$ transformations.	
\begin{lem}\label{lem:allowedFacts}Let $J$ be a subset of $S_{\alpha}$, let $h, j$ be elements of $S_{\alpha}$ and let  $i, k$ be in $\{1, \dots, r\}$. The following assertions hold.
\begin{enumerate}
\item\label{item:all1} If $J$ is allowed, the partitions $\alpha$ and $q_J(\alpha)$ have the same multiplicities for non zero values.
\item\label{item:all2} If $J$ is allowed and $2\not \in J$ or $\lambda_2 > 1$, then for all $J\subseteq I\subseteq S_{\alpha}$ we have \[q_I(\alpha) = q_{I\setminus J}(q_J(\alpha)).\]
\item\label{item:all3} If $J$ is allowed and $\lambda_2 = 1$ and $2\in J$ then $q_I(\alpha) = q_{(I\setminus J) -1 }(q_J(\alpha))$.
\item\label{item:all4} If $J$ is allowed, let $I$ and $J$ be subsets of $S_{\alpha}$ and $S_{q_J(\alpha)}\setminus J$ respectively, of cardinal $k$ and $k-|J|$ respectively. If $2\not \in J$ or $\lambda_2 > 1$ assume $I'\not = I \setminus J$. Otherwise $I'\not = (I\setminus J) -1$. Then the partitions $q_I(\alpha)$ and $q_I'(q_J(\alpha))$cannot be compared. 
\end{enumerate}
\end{lem}	
	\begin{prop}\label{fact:PurExt} There is a non zero morphism from $\PRes_{\alpha}$ to $\PRes_{q_J(\alpha)}[i]$ in the derived category if and only if $J$ is allowed and $|J| = i$.
	\end{prop}
	\begin{proof}Clearly it is true that $q_J(\alpha)\in [f(q_J(\alpha)), q_J(\alpha)]$. By Theorem \ref{lem:homsBool} there is an extension of degree $|J|$ if and only if there is no other subset $J'$ such that $q_{J'}(\alpha)\in [f(q_J(\alpha)), q_J(\alpha)]$. So we want to show that this is equivalent to $J$ being allowed. Assume $J$ is allowed and take $J'\not = J$. Because the antichain is \defn{boolean} it suffices to consider $J'\supset J$. Let $j$ be in $J'\setminus J$. Then 
\[(q_{J'}(\alpha))_{x_j} = \lambda_j - 1 < \lambda_j = (q_J(\alpha))_{x_j}\]
where $x_j$ is the end of the coefficient $\lambda_j$ in $\alpha$. Because $J$ is allowed $x_j$ is also the end of the coefficient $\lambda_i$ in $q_J(\alpha)$. By condition \ref{item:2} of Proposition \ref{lem:hom0} the inequality above implies that $q_{J'}(\alpha)\not\in[f(q_J(\alpha)), q_J(\alpha)]$. Reciprocally, assume that $J$ is not allowed. Then there exists $j\in S_{\alpha}$ such that $j\in J$, $j-1\not \in J$ and $\lambda_{j-1} = \lambda_j- 1$. One can then check that $q_{J\sqcup\{j-1\}}(\alpha) \in [f(q_J(\alpha)), q_J(\alpha)]$ using the characterisation for the support of the $\PRes_{\alpha}$ and conclude with Proposition \ref{lem:hom0}. Figure \ref{fig:PurExt} provides an illustration for the arguments of the proof.
\end{proof}
\begin{figure}[h!]
	\begin{minipage}{0.45\textwidth}\centering
		\begin{tikzpicture}[scale =2]
		\draw[step = .5cm, thin, lightgray] (0, 0) grid (3, 2);
		\draw[thick] (0, .5) -- (1, .5) -- (1, 1) -- (2, 1) -- (2, 2) -- (3, 2);
		\draw[thick, gray](0, 0.45) -- (1.05, .45)-- (1.05, 0.95) -- (2.05, .95) -- (2.05, 1.5) -- (3, 1.5);
		\draw[thick, lightgray](0, 0) -- (1.1, 0)-- (1.1, 0.9) -- (2.1, .9) -- (2.1, 1.45) -- (3, 1.45);
		\draw[thick, gray, dashed](0, 0) -- (.5, 0) -- (.5, .45) -- (1.5, .45) -- (1.5, .95) -- (3, .95) -- (3, 1.5);
		\draw[thick](3.02, 1.7)--(3.02, 2.1);
		\draw[thick, gray] (3.05, 1.7)--(3.05, 2.1);
		\draw[->] (2.25, 1.9)--node[right]{$J = \{3\}$}(2.25, 1.6);
		\draw[->] (.8, .4)--node[right]{$J' = \{1, 3\}$}(.8, .1);
		\draw[red] (.4, 0.05) rectangle (.6, 0.55);
		\end{tikzpicture}
		\caption*{$q_{\{1, 3\}}(\alpha)\not \in [f(\alpha_{\{3\}}), \alpha_{\{3\}}]$}
	\end{minipage}
	\begin{minipage}{0.45\textwidth}\centering
		\begin{tikzpicture}[scale =2]
		\draw[step = .5cm, thin, lightgray] (0, 0) grid (3, 2);
		\draw[thick] (0, .5) -- (1, .5) -- (1, 1) -- (2, 1) -- (2, 2) -- (3, 2);
		\draw[very thick, lightgray] (0, 0) -- (1, 0) -- (1, .5) -- (2.1, .5) -- (2.1, 1.9) -- (3, 1.9);
		\draw[thick, gray](0, 0.45) -- (2.05, 0.45) -- (2.05, .95) -- (2.05, 1.95) -- (3, 1.95);
		\draw[thick, gray, dashed](0, 0) -- (1.5, 0) -- (1.5, .45) -- (3, .45) -- (3, 1.95);
		\draw[thick](3.02, 1.7)--(3.02, 2.1);
		\draw[thick, gray] (3.05, 1.7)--(3.05, 2.1);
		\draw[->] (1.2, .9)--node[right]{$J = \{2\}$}(1.2, .6);
		\draw[->] (.2, .4)--node[right]{$J' = \{1, 2\}$}(.2, .1);
		\draw[red] (1.5, 0.5) circle (3pt);
		\end{tikzpicture}
		\caption*{$q_{\{1, 2\}}(\alpha) \in [f(\alpha_{\{2\}}), \alpha_{\{2\}}]$}
	\end{minipage}
	\newline\centering
	\begin{tikzpicture}
		\draw[thick] node [label=left:Legend:]{} (0, 0) -- (0.5, 0) node [right]{$\alpha$};
		\draw[very thick, lightgray] (2.5, 0) -- (3, 0) node [right, black]{$q_{J'}(\alpha)$};
		\draw[thick, gray] (5, 0) -- (5.5, 0) node [right, black]{$q_J(\alpha)$};
		\draw[thick, gray, dashed] (7.5, 0) -- (8, 0) node [right, black]{$f(q_J(\alpha))$};
	\end{tikzpicture}
	\caption{Illustration of the proof of Fact \ref{fact:PurExt} }\label{fig:PurExt}
\end{figure}
\begin{nota}\label{nota:elemExt}
There exists a \defn{canonical realisation} of these extensions defined as morphisms of complexes between the projective resolution associated to $\alpha$ and the interval module associated to $q_J(\alpha)$ concentrated in degree $|J|$. We take the morphisms of modules in degree $|J|$ to be the canonical projection of the factor $P_{q_J(\alpha)}$ onto $[f(q_J(\alpha)), q_J(\alpha)]$ and write the resulting morphism of complexes $u_J^{\alpha}$. Among the extensions we just exhibited we are particularly interested in the ones where $J = \{i\}\subset S_{\alpha}$ and write these $\prescript{1}{}{u_i^{\alpha}}$.  
\end{nota}
Just like for the degree zero morphisms, up to homotopy there exists a unique lift of these morphisms along the projective resolution $\PRes_{q_J(\alpha)}[|J|]$.
\begin{lem}\label{lem:explicitLift}Let  $\alpha$ be a left enhanced partition and let $J$ be an allowed subset of $S_{\alpha}$ of size $j$. The lift $\phi$ of $u^{\alpha}_J : \PRes_{\alpha} \to \PRes_{q_J(\alpha)}[|J|]$ has support $\bigoplus_{J\subseteq I} P_{q_I(\alpha)}$ in each degree. More precisely, in degree $j+k$, we have 
\[\phi_{j+k} = \bigoplus_{\substack{J\subset I\\ |I| = j+k}} (-1)^{|I\setminus J|_J + |J|\cdot k}\cdot id_{P_{q_I(\alpha)}}\]
where $|I|_J = \sum_{i\in I}|i|_J$.
\end{lem}

\begin{proof}The morphism of complexes $\phi = (\phi_l)_l$ is concentrated in degrees greater than the cardinal $j$ of $J$. In each degree, because the antichains are strong, the morphism of modules $\phi_{j+k}$ decomposes as linear combination of the identity morphisms of the indecomposable projective summands appearing in both the source and the target. These are precisely the summands $P_{q_I(\alpha)}$ satisfying $J\subseteq I$ as recorded in Lemma \ref{lem:allowedFacts}. It remains to determine the coefficients $\varepsilon(I)$ associated to each such summand by induction . Assume that $\lambda_2 > 1$ or that $2\not \in J$. The relevant morphisms fit in the following commutative square.
\begin{equation}
\begin{tikzcd}[column sep = huge]
P_{q_{I\sqcup \{i\}}(\alpha)}\ar[r, "(-1)^{|i|_{I\sqcup \{i\}}}"]\ar[d, "\varepsilon(I\sqcup\{i\})"']&P_{q_I(\alpha)}\ar[d, "\varepsilon(I)"]\\
P_{q_{I\setminus J \sqcup \{i\}}(q_J(\alpha))}\ar[r, "(-1)^{j+|i|_{I\setminus J\sqcup \{i\}}}"]& P_{q_{I\setminus J}(q_J(\alpha))}
\end{tikzcd}
\end{equation}
Recall from equation (\ref{eq:boundComp}) that $|i|_I = \{h\in I | h\leq i\} = |i|_{I\setminus J} + |i|_{J}$. 
Hence we get a recursive formula for the coefficient $\varepsilon(I\sqcup \{i\}) = (-1)^{j + |i|_J}\times \varepsilon(I)$.
By setting the degree $j$ sign to be $1$, and taking the product the previous coefficients, we get the claimed formula. If $\lambda_2 = 1$ and $2\in J$ the argument is similar. See Lemma \ref{lem:allowedFacts} for the relevant technical difference. To conclude we underline that $\varepsilon(I)$ is indeed not zero for all subset $I$ of $S_{\alpha}$ containing $J$.
\end{proof} 

We can now express the existence of morphisms in terms of these special extensions and our characterisation of degree zero morphisms and deduce a factorisation of any morphism into a special extension followed by a degree zero morphism.
\begin{prop}\label{prop:factHomYildCat}Let $\phi : \PRes_{\alpha}\to \PRes_{\beta}[i]$ be a non zero morphism in $\mathcal{Y}_{m,n}$. Then there exists a unique subset $J$ of $S_{\alpha}$ such that $\phi$ factors through $\PRes_{q_J(\alpha)}[i]$ and $|J| = i$ completing the following commutative diagram.
\[
\begin{tikzcd}[column sep = tiny]
\PRes_{\alpha}\ar[rr]\ar[dr]&&\PRes_{\beta}[|J|]\\
&\PRes_{q_J(\alpha)}[|J|]\ar[ur]&
\end{tikzcd}
\]
\end{prop}
\begin{proof}
By Theorem \ref{lem:homsBool}, there exists a non zero morphism $\phi$ if and only if there exists a unique subset $J$ such that $q_J(\alpha)\in [f(\beta), \beta]$. Moreover any two morphisms between $\PRes_{\alpha}$ and $\PRes_{\beta}[i]$ differ by multiplication by a scalar. Using Lemma \ref{lem:allowedFacts}, we can show that the following morphisms exist and are unique up to multiplication by a scalar
	$$\PRes_{\alpha}\to \PRes_{q_J(\alpha)}[|J|]\to\PRes_{\beta}[|J|].$$
%The argument relies of the observation that the existence of $\phi$ imposes a special condition on the first value of the partition $\beta$. 
We refer the reader to \cite[Lemma 5.1.11]{gottesman2024these} for a full argument. 
It remains to check that their composition is non zero in the homotopy category. We represent the morphisms in the following diagram.
\[
	\begin{tikzcd}
		\PRes_{\alpha}:\ar[d]\ar[dd, bend right = 50, dashed, "?"]&\dots\ar[r]\ar[d]&\bigoplus_{|I| = |J|}P_{\alpha_I}\ar[d]\ar[r]&\bigoplus_{|I| = |J| - 1 }P_{\alpha_I}\ar[r]\ar[d]&\dots\ar[r]&P_{\alpha}\\
		\PRes_{q_J(\alpha)}[|J|]\ar[d]:&\dots\ar[r]\ar[d]&P_{q_J(\alpha)}\ar[r]\ar[d]&0\ar[d]&&\\
		\PRes_{\beta}:&0\ar[r]&{[f(\beta), \beta]}\ar[r]\ar[from = uu, bend right = 40, crossing over, dashed, "?"']\ar[from = uur, crossing over, dashed, "0"']&0&&
	\end{tikzcd}
\]
The composition of the two morphisms of modules in degree $|J|$ is non zero given the support of the top map as described in Lemma \ref{lem:explicitLift}. No homotopy can be constructed because by assumption if $I\not= J$, $q_I(\alpha)\not\in [f(\beta), \beta]$. By multiplying one of the intermediate maps by a unit we can indeed factor $\phi$ through $\PRes_{q_J(\alpha)}[i]$.
\end{proof}
We now decompose further the extensions $\PRes_{\alpha} \to \PRes_{q_J(\alpha)}[|J|]$.
\begin{lem}\label{lem:decompExt}Let $\alpha$ be a partition and $J$ be an allowed subset of $S_{\alpha}$. Then the extension $u_{J}^{\alpha}:\PRes_{\alpha}\to\PRes_{q_J(\alpha)}[k]$ discussed in Proposition \ref{fact:PurExt} decomposes as
	\[\PRes_{\alpha}\xrightarrow{\prescript{1}{}{u_{j_1}}}\PRes_{q_{\{j_1\}}(\alpha)}[1]\xrightarrow{\prescript{1}{}{u_{j_2}[1]}}\dots\xrightarrow{\prescript{1}{}{u_{j_k}}[k-1]} \PRes_{q_J(\alpha)}[k],\]
where $J = \{j_1, \dots j_k\}$ is ordered by $j_t < j_{t+1}$.
\end{lem}
\begin{proof}
First notice that the truncated sets $J_i = \{j_1, \dots, j_i\}$ are all allowed. We proceed by induction on $k$. Assume the result holds for $k-1$. Then by Theorem \ref{lem:homsBool} it suffices to show that the composition 
\[\PRes_{\alpha}\to\PRes_{q_{J_{k-1}}(\alpha)}[k-1] \to \PRes_{q_J(\alpha)}[k]\]
is non zero. Again we can draw a diagram to visualise the situation:
\[
\begin{tikzcd}
\PRes_{\alpha}:\ar[d]\ar[dd, bend right = 60, dashed, "?"' near start]&[-35pt]\dots\ar[r]\ar[d]&\bigoplus_{|I| = k}P_{q_I(\alpha)}\ar[r]\ar[d]&[-25pt]\bigoplus_{|I| = k-1}P_{q_I(\alpha)}\ar[d]\ar[r]&[-30pt]\bigoplus_{|I| = k-2}P_{q_I(\alpha)}\ar[d]\ar[r]&[-20pt]\dots\\
\PRes_{q_{J_{k-1}}(\alpha)}[k-1]: \ar[d]&\dots\ar[d]\ar[r]&\bigoplus_{i\in S_{q_{J_{k-1}}(\alpha)}}P_{q_{i}(q_{J_{k-1}}(\alpha))}\ar[d]\ar[r]&P_{q_{J_{k-1}}(\alpha)}\ar[d]\ar[r]&0&\\
\PRes_{q_J(\alpha)}[k]:&0\ar[r]&{[f(q_J(\alpha)), q_J(\alpha)]}\ar[r]\ar[from = uu, bend right = 75, crossing over, dashed, "?"']\ar[from = uur, bend left = 10, crossing over, dashed, "0"]&0&&
\end{tikzcd}
\]
It is clear that $q_I(\alpha)\not\leq q_J(\alpha)$ when $|I| = k-1$. It might not be clear however that the composition of the two morphisms of module in degree $k$ is non zero. Because the bottom map has support $P_{q_J(\alpha)}$ we need to argue that the restriction of the top map to $P_{q_J(\alpha)}$ is non zero either. This is the case by Lemma \ref{lem:explicitLift}. 
Finally, we need to point out that this results in the canonical map $u_J$ and not just a map proportional to it meaning that no negative sign appears when composing the two maps. We compute the sign of the module map in degree $j$ with Lemma \ref{lem:explicitLift}. It has support the indecomposable projective module $P_{q_J(\alpha)}$. Its sign is 
\[(-1)^{|j_k|_{J_{k-1}} + 1\cdot(k-1)} = (-1)^{2\cdot(k-1)} = 1\]
Hence no sign appears.
\end{proof}
Finally, we give a similar decomposition for degree zero morphisms.
\begin{lem}\label{lem:decomHom0}
Let $\phi : \PRes_{\alpha}\to\PRes_{\beta}$ be a non zero morphism of modules. Then there exists a sequence $(d_1, \dots, d_r)$ such that for all $0\leq i< r $, we have $0\leq d_i\leq \mu_{i}$, with $d_i = \mu_i$ only if $i = 1$ and $\lambda_1 = 0$, 
	\[\beta = p_1^{d_1}\circ\dots\circ p_r^{d_r}(\alpha).\]
Moreover, the morphism factors through each of the objects associated to the intermediate partitions.
\end{lem}
\begin{proof}By Theorem \ref{lem:homsBool} and Proposition \ref{lem:hom0}(\ref{item:5}) the inclusion 
\[\{\lambda_i | i\in S_{\alpha}\}\subseteq \{l_j | j\in S_{\beta}\}\]holds. In other words, for all $i\in S_{\alpha}$, there exists $j\in S_{\beta}$ such that $\lambda_i = l_j$. Moreover their ending indices satisfy $x_{i-1}<y_j \leq x_i < y_{j+1}$. For $i\in S_{\alpha}$ we set $d_i = x_i-y_j$. If $l_1 = 0$, because $\alpha \leq \beta$ we also have $\lambda_1 = 0$ and $y_1\leq x_1$ so set $d_1 = x_1 - y_1$. The condition $d_i< \mu_i$ holds because $x_{i-1} < y_{j}$ when $i\not= 1$. The resulting partition has the same coefficients as $\beta$. The multiplicities also match when $i\leq r$:
\[x_{i} - d_i - x_{i-1} - d_{i-1} = x_{i} - x_{i} + y_{j} - x_{i-1} +x_{i-1} - y_{j-1} = m_j.\]
To see that the map factors through the intermediate partitions use Proposition \ref{lem:hom0}(\ref{item:4}): 
let $k$ be in $S_{\alpha}$, $0\leq d < d_k$ and set $\gamma = p_k^d\circ p_{k+1}^{d_{k+1}}\circ\dots\circ p_{r}^{d_r}(\alpha)$ and assume the morphism $\PRes_{\alpha} \to \PRes_{\gamma}$ can be decomposed through all the intermediate links. By construction, $\alpha\in [f(p_k(\gamma)), \gamma]\cap [f(\gamma), \gamma]\cap [f(\alpha), \alpha]$ so the map induced by the composition 
\[\PRes_{\alpha}\to \PRes_{\gamma}\xrightarrow{\prescript{0}{}{u_k^{\gamma}}}\PRes_{p_{k}(\gamma)}\]
is non zero on this vertex hence the composition itself is non zero. Recall that by Theorem \ref{lem:homsBool} the hom spaces are one dimensional. It follows that the map $\phi$ can be factored into a sequence of $\prescript{0}{}{u_k^{\alpha'}}$ maps.
\end{proof}
\begin{figure}[h!]
	\begin{minipage}{0.45\textwidth}\centering
		\begin{tikzpicture}[scale =2]
		\draw[step = .5cm, thin, lightgray] (0, 0) grid (3, 2);
		\draw[thick] (0, .5) -- (1, .5) -- (1, 1) -- (2, 1) -- (2, 2) -- (3, 2);
		\draw[thick, lightgray](0, .45) -- (1.5, .45) -- (1.5, .95) -- (2.05, 0.95) -- (2.05, 1.95) -- (3, 1.95);
		\draw[thick, dashed] (0, 0) -- (.5, 0) -- (.5, .5) -- (1.5, .5) -- (1.5, 1) -- (3, 1) -- (3, 2);
		%\draw[thick](3.05, 1.7) -- (3.05, 2.1);
		%\draw[thick, gray] (1.95, 1.7) -- (1.95, 2.1);
		\draw[<-] (1.1, .65)--node[above]{$d_1$}(1.4, .65);
		%\draw[->] (2.9, 2.05)--node[below]{$d_r$}(2.1, 2.05);
		\end{tikzpicture}
		\caption*{$\alpha_{x_{i-1}} = \beta_{x_{i-1}}$}
	\end{minipage}
	\begin{minipage}{0.45\textwidth}\centering
	\begin{tikzpicture}
		\draw[thick] node [label=left:Legend:]{} (0, 0) -- (0.5, 0) node [right]{$\beta$};
		\draw[thick, dashed] (2, 0) -- (2.5, 0) node [right, black]{$f(\beta)$};
		\draw[thick, lightgray] (4, 0) -- (4.5, 0) node [right, black]{$\alpha$};
		\draw (1.5, -1) -- node[]{$d_0 = 0, d_1 = 1, d_2 = 0$} (1.5, -1);
	\end{tikzpicture}
	\end{minipage}
	\caption{Illustration of the proof of Lemma \ref{lem:decomHom0}} \label{fig:Factorisation12}
\end{figure}
Morphisms in the category $\mathcal{Y}_{m,n}$ can thus be decomposed into the morphisms of degree one and the morphisms of degree zero we have exhibited. Most of them cannot be decomposed further.
\begin{lem}\label{lem:indec}
Let $\alpha$ be a partition, $h\in S_{\alpha}$ allowed and $k\in \{1, \dots, r\}$ such that $p_k$ is well defined. Assume that $h \not = \epsilon$ or that $ \lambda_2 > 1$. Then the morphisms $\prescript{1}{}{u_h}$ and $\prescript{0}{}{u_k}$ are irreducible in $\mathcal{Y}_{m,n}$.
\end{lem}

\begin{proof}
Consider a morphism in degree zero $\prescript{0}{}{u_k}$ and assume it factors through an object $\PRes_{\beta}$. The partitions $\alpha$ and $p_k(\alpha)$ are identical except in position $x_i$. By Lemma \ref{lem:hom0}(\ref{item:5}) it holds that $\alpha \leq \beta \leq p_k(\alpha)$ and that $\beta$ has the same coefficients as $\alpha$ and $p_k(\alpha)$. Hence $\beta = \alpha$ or $\beta = p_k(\alpha)$.
\par Next let $\prescript{1}{}{u_h}$ be an extension with $\alpha$ and $\beta$ as before. If the object $\PRes_{\beta}$ is shifted by $0$, $\alpha$ and $\beta$ satisfy Lemma \ref{lem:hom0}(\ref{item:5}). Simultaneously by Theorem \ref{lem:homsBool} and Proposition \ref{prop:factHomYildCat} there exists $f$ such that $q_f(\beta)$ and $q_i(\alpha)$ satisfy Lemma \ref{lem:hom0}( \ref{item:5}). It follows that both the coefficients and the multiplicities of $\alpha$ and $\beta$ match. The case where $\PRes_{\beta}$ is shifted by one is similar and is left to the reader. More details can be found in \cite[Lemma 5.1.15]{gottesman2024these}
\end{proof}
\begin{ex}\label{note:extNonElem} If $\lambda_2 = 1$, then $\prescript{1}{}{u_2}$ decomposes into $\prescript{1}{}{u_1}\circ\prescript{0}{}{u_1^{\mu_1-1}}\circ\prescript{0}{}{u_0}$ where $\mu_1$ is the multiplicity of the value zero in the source partition. We give a concrete example below.
\[
\begin{tikzcd}[column sep = tiny]
\PRes_{(0011)}\ar[rr]\ar[dr]&&\PRes_{(0000)}[1]\\
&\PRes_{(1111)}[1]\ar[ur]&
\end{tikzcd}
\]
\end{ex}
%
%%%%%%%%%%%%%%%%%%%%%%%%%%%%%%%%%%%%%%%%%%%%%%%%%%%%%%
\subsection{Configurations and relations}
\label{subsection:configAndRels}
We can now describe the irreducible  morphisms of subsection \ref{subsec:homsets1} using configurations through the map $\phi_r$  from subsection \ref{subsec:grids}. To do so we define a partial function on configurations, as in \cite{Herschend_2011} for the construction of higher Auslander algebras of type A. If $R$ is a configuration, $k \in R$ and $k - 1 \not\in R$, we write 
\begin{equation}\sigma_k^-(R) = (R\cup{\{k-1\}})\setminus\{k\}.\end{equation}
%
%Note that if $k = -m$, then $\sigma^-_k(R) = (R\cup{\{n\}})\setminus\{-m\}$ as $k$ represents an element of $\mathbb{Z}/(m+n+1)\mathbb{Z}$. From now on, we will also write the objects of the category $\mathcal{Y}_{m,n}$ as $\PRes_{R}$, implicitly identifying $\alpha$ with $R$ using the bijection $\phi_r$.  
Because we only consider plain partitions, the abaci we will encounter have no bead in column $k = -m$. Hence we will ont consider transformations $\sigma_{-m}$ and $\sigma_{-m+1}$. However, the proofs of this subsection hold without his limitation.

\begin{ex} Recall the right abacus from Example \ref{ex:confToAbac} associated to the partition $a = (0, 2, 3, 7, 7)$.
\begin{center}\begin{tabular}{c c c c c | c c c c c c c c}
-5&-4&-3&-2&-1&0&1&2&3&4&5&6&7 \\
\hline
&$\bullet$&&&&$\bullet$&&$\bullet$&$\bullet$&&&&$\bullet$
\end{tabular}
\end{center}
Applying $\sigma^-_{0}$ to the configuration associated to $a$ consists of taking the bead placed in $-4$ and sliding it down to the $-5$. We then get the following abacus,
\begin{center}\begin{tabular}{c c c c c | c c c c c c c c}
-5&-4&-3&-2&-1&0&1&2&3&4&5&6&7 \\
\hline
&$\bullet$&&&$\bullet$&&&$\bullet$&$\bullet$&&&&$\bullet$
\end{tabular}
\end{center}
which is associated to the partition $(2,2,3,7,7)$.
\end{ex}

\begin{prop} \label{prop:configArrows}
In terms of configurations, the irreducible morphisms in degree zero, correspond to arrows
\[\PRes_R\to \PRes_{\sigma^-_k(R)}\text{ with } 0 \geq k > -m+1 \in R.\]
Next the irreducible morphisms in degree 1 correspond to the following transformation on configurations
\[\PRes_R\to \PRes_{\sigma^-_k(R)}\text{ with } 0 < k\in R.\]
\end{prop}
Note that, for a partition $\alpha$ with $\lambda_2 = 1$, the extension $\prescript{1}{}{u_2^{\alpha}}$ does not appear in the statement as the transformation $q_2$ of the partition does not correspond to a well defined transformation $\sigma_1^-$ of the corresponding configuration. This is coherent with Example \ref{note:extNonElem}.
\begin{proof}First we prove the statement about the degree zero irreducible morphisms. Let $\alpha$ be a partition. Recall from Corollary \ref{nota:Hom0} and Lemma \ref{lem:decomHom0} that these morphisms were of the form
$$\PRes_{\alpha} \to \PRes_{p_i(\alpha)}$$with $i\in \{1, \dots, r\}$ and $\mu_i > 1$. Using \ref{nota:Hom0}, we compute 
\begin{equation}
x_j^{p_i(\alpha)} = \begin{cases} x_{i} - 1 &\text{ if } j = i\\
						x_j &\text{ otherwise. }
\end{cases}\end{equation}This implies that $R_{p_i(\alpha)} = (R_{\alpha}\cup\{-x_i\}) \setminus\{-x_i + 1\} = \sigma_{-x_i + 1}^-(R_{\alpha})$. Finally the degree 1 morphisms were $$\PRes_{\alpha} \to \PRes_{q_{i}(\alpha)}[1]$$ where $\lambda_i -1> \lambda_{i-1}$. Here the configuration associated $\alpha_{\{i\}}$ is clearly the configuration associated to $\alpha$ where the $\lambda_i$ was replaced by $\lambda_i - 1$. 
\par Reciprocally, every $k\in R$ such that $\sigma^{-}_k(R)$ is well defined, corresponds to either an irreducible morphism or  an irreducible extension through the inverse of $\phi_r$ (recall Figure \ref{fig:exPhirInverse}). If $k \leq 0$, then $\sigma^{-}_k(R)$ corresponds to the transformation $p_i(\alpha)$ if $k -1$ is the $i^{th}$ gap in the negative side counting from the right. It is associated to $\prescript{0}{}{u_i}$. If $k > 0 $, then $\sigma^{-}_k(R)$ corresponds to the transformation $q_j(\alpha)$ if $k$ is the $j^{th}$ positive element of the configuration, in increasing order, hence it is associated to $\prescript{1}{}{u_j}$. This conclude the proof.
\end{proof}
Note that configurations make the description of the irreducible morphisms more homogeneous. This will also be the case for the relations between them. To simplify notations, the morphism (in degree zero or one) from $\PRes_{R}$ to $\PRes_{\sigma_k^-(R)}$ will be denoted $u_k^R$ since the degree of the morphism is encoded in the sign of $k$. 
We will now express the relations between the morphisms and extensions in the language of configurations.
	\begin{lem}\label{lem:RelIdentification}
	Let $R$ be a configuration, take $k, l\in R$ such that $k - 1, l - 1 \not\in R$. We have the following equalities in the category  $\mathcal{Y}_{m,n}$: 
\begin{equation}\label{eq:rel}
\rho_{k, l}^R :  u_l^{\sigma_k^-(R)}\circ u_{k}^R = \varepsilon u^{\sigma^-_l(R)}_k \circ u_l^R \text{ and }z_k^R : u_{k-1}^{\sigma^-_{k}(R)}\circ u_{k}^R = 0
\end{equation}
where $\varepsilon = -1$ if $k$ and $l$ are positive and $1$ otherwise.
	\end{lem}
	\begin{proof}
We distinguish several cases depending on the sign of the integers $k$ and $l$. First consider $\rho_{k, l}^R $ when both $k$ and $l$ are positive \emph{i.e.} the irreducible morphisms are concentrated in degree 1. Then the morphisms in question can be summed up in the following diagram. We use partitions because the order relation is not clear on configurations. The elements $k$ and $l$ of $R$ can be uniquely associated to $i, j\in S_{\alpha}$
		\begin{equation}\label{diag:proofRel1-1_1}
		\begin{tikzcd}
		\PRes_{\alpha}\ar[d]\ar[dd, bend right = 40, dashed, "?"']:&[-30pt] \dots\ar[d]\ar[r]&\bigoplus_{i', j'\in S_{\alpha}}P_{q_{\{i',j'\}}(\alpha)}\ar[d]\ar[r]&[-20pt]\bigoplus_{i'\in S_{\alpha}}P_{q_{i'}(\alpha)}\ar[d]\ar[r]&P_{\alpha}\ar[d]\ar[r]&[-10pt]0\\
		\PRes_{q_i(\alpha)}\ar[d]:&\dots\ar[d]\ar[r]&\bigoplus_{j'\in S_{\alpha}}P_{q_{j'}(q_{i}(\alpha))}\ar[d]\ar[r]&P_{q_i(\alpha)}	\ar[d]	\ar[r]&0&\\
		\PRes_{q_{\{i, j\}}(\alpha)}:& 0\ar[r]& \left [f(q_{\{i, j\}}(\alpha)), q_{\{i, j\}}(\alpha)\right] \ar[r]\ar[from = uu, bend left = 50, dashed, 			"?", 	crossing over]&0&&
		\end{tikzcd}
		\end{equation}
	The resulting morphism of modules in degree two is the morphism from $\PRes_{\alpha}$ to $\PRes_{q_{\{i, j\}}(\alpha)}$ described in \ref{fact:PurExt} up to a sign because $i$ and $j$ are assumed to be allowed. This is symmetric in $i$ and $j$ hence the result. To be more specific, according to Lemma \ref{lem:explicitLift}, signs appears in the upper middle square. The boundary map component going from $P_{q_{i, j}(\alpha)}$ to $P_{q_{\{i\}}(\alpha)}$ has sign $(-1)^{|j|_{\{i, j\}}}$. Hence the two compositions differ by a factor $-1$. 
The remainder of the proof is similar and can be found in \cite[Lemma 5.2.3]{gottesman2024these}.
	\end{proof}
\begin{rem}\label{rem:signPbms}The fact that some squares commute and other square anticommute is inconvenient for Theorem \ref{thm:DerEqTypeA}. However, in Corollary \ref{cor:Pres}, we will see that these signs can be corrected.
\end{rem}
\begin{Def}\label{def:allowedConfig} A subset $J$ of a configuration $R$ is allowed if for all $j\in J$ either $j-1 \not \in R$ or $j-1\in J$. Then $\sigma_J^-(R)$ is well defined. Note that if $J$ is in the positive side of $R$, $J$ corresponds to an allowed subset of $S_{\alpha}$.
\end{Def}
\begin{lem}\label{lem:zeroRelsOneSide}Let $R_0, \dots, R_t$ be configurations such that for all $q\in\{0, \dots, t-1\}$, there exists $k_q\in R_q$ such that $R_{q+1} = \sigma^-_{k_q}(R_q)$. Then the morphism $f = u_{k_{p-1}}^{R_{p-1}}\circ \dots \circ u_{k_0}^{R_0}$ is non zero \emph{only if} for all $q$, $k_q\in R_0$. 
\end{lem}
\begin{proof}
 Assume that there exists $q$ such that $k_q\not\in R_0$. Then there exists $r< q$ with $k_r = k_q+1$. Take $q$ to be minimal and $r$ as close as possible to $q$. Take $s \in \{r, \dots, q\}$. Then, by the assumptions on $q$ and $r$, we either have $k_s > k_q+1$ or $k_s < k_q$. Whenever $s$ is such that $k_{s+1} < k_q < k_q+1 < k_{s}$, equation (\ref{eq:rel}) provides the equality
\[u_{k_{s+1}}\circ u_{k_s} = u_{k_s}\circ u_{k_{s+1}}.\]
Hence using equation (\ref{eq:rel}) we can rewrite the morphism $f$ as 
\[f = \epsilon \times ( \dots \circ u_{k_q}\circ u_{k_{s(q-1)}} \circ\dots\circ u_{k_{s(\tau+1)}} \circ u_{k_{s(\tau)}}\circ\dots\circ u_{k_{s(r+1)}} \circ u_{k_r} \circ\dots)\]
where $s$ is a permutation on the set $\{r+1, \dots , q-1\}$ and if $i > \tau$, $k_{s(i)} > k_q+1$ and $k_{s(i)} < k_q$ otherwise. Finally, we use equation (\ref{eq:rel}) again to get
\[f = \epsilon \times ( \dots \circ u_{k_{s(q-1)}} \circ\dots\circ u_{k_{s(\tau+1)}}\circ u_{k_q}\circ u_{k_r}  \circ u_{k_{s(\tau)}}\circ\dots\circ u_{k_{s(r+1)}} \circ\dots)\]
The central term is zero by Lemma \ref{lem:RelIdentification}. This concludes the proof. 
\end{proof}
We now want to find a canonical way of decomposing morphisms into irreducible ones. Knowing that irreducible morphisms are indexed by certain elements of the partition $R$, we want to equip $R$ with a convenient order relation. Recall that when we defined configurations on $\mathcal{Z}$ we wrote them as increasing sequences, by equipping $\mathcal{Z}$ with the order relation on $\mathbb{Z}$. This ordering is not adapted to the study of the morphisms: for any partition $R$ containing $-m$ but not $n$, it is in general unclear whether the morphism $u^R_{-m}$ should come before or after $u^R_{k}$ when $k$ is positive. Let $R$ and $R'$ be configurations. Suppose that $R' = \sigma_{k_t}\circ\dots\circ\sigma_{k_1}(R)$, that for each $q\leq t$, $k_q\in R$ and the partition $R_q = \sigma_{k_q}\circ\dots\circ\sigma_{k_1}(\alpha)$ is well defined. We have the following composition of irreducible morphisms 
\[f = u_{k_t}^{R_{t-1}}\circ\dots\circ u_{k_1}^R\]
in $\mathcal{Y}_{m,n}$. Order the combinatorial data $K = (k_1, \dots, k_p)$ and by the same process the elements of $\mathcal{Z}$ as follows: let $k_{min}$ be the maximal element of $\mathcal{Z}$ for the naive order which  does not appear in $K$ and is less than or equal to $1$. We consider the total order
 \[k_{min} <_f k_{min} + 1 <_f\dots<_f n <_f -m <_f\dots <_f k_{min} - 1\]
Because we picked the combinatorial data so that the consecutive morphisms are well defined, if $k_{min} \not = 1$ then $k_{min}$ is not an element of $R$. Alternatively, if $k_{min} = 1$, either $1\in R$ and the first value of the list $K$ ordered with $<_f$ is $k >_f 2$ or $1\not\in R$ and $k >_f 1$. Hence, if $k$ is the minimum of $K$ with regard to $<_{f}$, then $k-1$ is not in $R$.
\begin{lem}\label{lem:ordering}Consider two configurations $R$, $R'$ and a morphism $f$ as above with combinatorial data $k_1, \dots, k_t$ and order $<_f$. Suppose also that for all $1\leq q \leq t$, $k_{q}$ is in $R$. Then there exists a permutation $p\in \mathfrak{S}_{t}$ such that we have 
\[k_{p(1)} <_f \dots <_f k_{p(t)}\]
and 
\[f = \pm u_{k_{p(t)}}^{R_{p(t-1)}}\circ\dots\circ u_{k_{p(1)}}^{R}.\]
\end{lem}
\begin{proof}Recall from equation (\ref{eq:rel}) of Lemma \ref{lem:RelIdentification}, that for each $q$, if $k_q, k_{q+1} \in R_{q-1}$ while $k_{q} - 1\not\in R_{q-1}$ and $k_{q+1} - 1\not\in R_{q-1}$ then
\begin{equation*}u_{k_{q+1}}^{R_{q}}\circ u_{k_{q}}^{R_{q-1}} = \epsilon u^{\sigma_{k_{q+1}}(R_{q})}_{k_{q}}\circ u_{k_{q+1}}^{R_{q-1}}\end{equation*}
with the transformation $\sigma_{k_q}$ being well defined on $R_{q-1}$. 
We want to show that this equation applies when $k_{q+1} <_f k_q$. The fact that $k_q -1 \not \in R_{q-1}$ follows from the assumption that $\sigma_{k_{q}}$ is well defined on $R_{q-1}$. Because $\sigma_{k_{q+1}}$ is well defined on $R_{q}$ and $k_{q+1} <_{f} k_q$ it must also be that $k_{q+1} - 1 \not \in R_{q-1}$. 
Using the bubble sort algorithm \cite{bubs}, the chain can be ordered up to a sign by applying relation $\rho_{k_{q+1}, k_q}$ when $k_q <_f k_{q-1}$. Indeed the bubble sort algorithm only swaps pairs of entries when they are not ordered according to the relation $<_f$. Moreover the assumptions of the lemma are maintained at each swap.
\end{proof}
We can now argue that we have indeed identified all the zero relations and prove the converse of Lemma \ref{lem:zeroRelsOneSide}.
\begin{lem}\label{lem:relZeroSolid}Let $R_0, \dots, R_t$ be configurations such that for all $q\in\{0, \dots, t-1\}$, there exists $k_q\in R_q$ such that $R_{q+1} = \sigma^-_{k_q}(R_q)$. Then the morphism $f = u_{k_{p-1}}^{R_{p-1}}\circ \dots \circ u_{k_0}^{R_0}$ is non zero if and only if for all $q$, $k_q\in R_0$. 
\end{lem}
\begin{proof} Assume that $k_q\in R = R_0$ for all $0 \leq q\leq t-1$. By Lemma \ref{lem:ordering} we can assume the elements of the chain are ordered with respect to the ordering $<_{f}$. Let $\tau$ be the maximal index such that $k_{\tau} \leq_{f} n$. First we show that $f_{ext} = u_{k_{\tau}}\circ\dots\circ u_{k_{1}}$ is a non zero extension. Because $k_{min}\not \in R $ or $k_{1} > 1$ (see remark before Lemma \ref{lem:ordering}), the set $\{k_{1}, \dots, k_{\tau}\}$ is allowed as per Definition \ref{def:allowedConfig} and see Subsection \ref{subsec:grids}. So $f_{ext}$ is non zero  in $\mathcal{Y}_{m,n}$ by Proposition \ref{fact:PurExt} and Lemma \ref{lem:decompExt}.
\par We now consider the morphism $f_{0} = u_{k_{t-1}}\circ\dots\circ u_{k_{\tau + 1}}$ which is concentrated in degree zero. To see that it is non zero we argue that its source and its target satisfy the condition of Lemma \ref{lem:hom0}(\ref{item:5}). Note that the source and the target of $f_0$ match in columns $k_{min}$ to $n-1$. Hence we have that $\{\lambda_i |i\in S_{\alpha_{\tau}}\} \subseteq \{l_j|j\in S_{\alpha'}\}$ completing the second part of Lemma \ref{lem:hom0}(\ref{item:5}). 
\par For the first part of Lemma \ref{lem:hom0}(\ref{item:5}), we need to argue that negative elements of the configurations that refer to the same coefficient in the partition interlace. It is easy to argue that the negative elements of the configuration interlace. Showing that interlacing elements correspond to the same coefficients is longer. The details are left to the reader, but can be found in \cite[Chapter 5.2]{gottesman2024these}.

\par The final step is to argue that $f = f_0\circ f_{ext}$ is non zero in $\mathcal{Y}_{m,n}$. We have shown that there exists $J\subseteq S_{\alpha}$ allowed such that there exists a non zero morphism from $\PRes_{q_J(\alpha)}$ to $\PRes_{\alpha'}$. Moreover, if $1 = \lambda_{\epsilon} \in R$ and $\epsilon\in J$, then $0$ is an element of $R'$ because the bead associated to $1$ can only be moved once. The same is true for all the beads associated to zero if it was an element of $R$. Hence if $\epsilon \in J$ and $\lambda_{\epsilon} = 1$, then the partition $\beta$ should have zero for its first value \emph{i.e.} $l_1 = 0$ and $y_1 > x_{\epsilon - 1}$. One can show that this indeed implies that there exists a morphism in the derived category from $\PRes_{\alpha}$ to $\PRes_{\alpha'}$. See \cite[Lemma 4.1.11]{gottesman2024these} for the details of the proof. Moreover, by Theorem \ref{lem:homsBool}, the set $J$ that we have identified is the unique subset of $S_{\alpha}$ such that $q_J(\alpha)\in [f(\alpha'), \alpha']$. By the proof of Proposition \ref{prop:factHomYildCat}, the composition $f = f_0\circ f_{ext}$ is non zero. This concludes the proof.
\end{proof}

\begin{rem}
By Proposition \ref{prop:factHomYildCat} and Lemmas \ref{lem:decompExt} and \ref{lem:decomHom0} all morphisms decompose into an extension post composed by a degree zero morphism. In turn these can be further decomposed into the irreducible morphisms we have identified. The list of the combinatorial data of the morphisms we thus obtain is called the \defn{canonical list} of the morphism. Its ordering coincides with $<_f$. Its positive elements are the elements of the unique subset $J$ of $S_{\alpha}$ such that $q_J(\alpha)\in [f(\beta), \beta]$ gives positive elements of the canonical list. Then, the inequalities between the ending indices described in Lemma \ref{lem:hom0}(\ref{item:5}) gives the remainder of the canonical list by looking at the transformation associated to the integers in the intervals $]-x_i, -y_j]$.
\end{rem}
As a consequence we have the following proposition
\begin{prop}\label{prop:RelYildCat} The relations described in equation (\ref{eq:rel}) generate the relations between morphisms in the category  $\mathcal{Y}_{m,n}$. Hence $\mathcal{Y}_{m,n}$ is generated by quadratic relations
\end{prop}
\begin{proof}
Consider the $\field$-linear category $C$ defined as follows:
\begin{itemize}
\item the objects of $C$ are pairs $(R, l)$ where $R$ is a configuration and $l$ is an integer;
\item the morphisms are generated by arrows $(R, l)\to(\sigma_k^-(R), l')$ with $l' = l + 1$ if $k>0$ and $l' = l$ otherwise,
\item with relations $\rho_{k,l}$ and $z_k$ identified in Lemma \ref{lem:RelIdentification}.
\end{itemize}
By Lemma \ref{lem:RelIdentification} again there is a well defined essentially surjective functor $F: C\to \mathcal{Y}_{m,n}$. To prove the current proposition we need to argue that it is an equivalence of categories. By Proposition \ref{prop:factHomYildCat}, Lemma \ref{lem:decomHom0} and Lemma \ref{lem:decompExt} combined with Proposition \ref{prop:configArrows} the functor induces surjective maps between the hom spaces $\Hom_{C}(R, R')$ and $\Hom_{D^b(J_{m,n})}(\PRes_R, \PRes_{R'})$. It remains to see that this map is injective. Consider an element 
\begin{equation}\label{eq:genrel}
\sum_{i=1}^q a_i\cdot f_i \text{ of } \Hom_{C}(R, R')
\end{equation}
where $q$ is an integer, $a_1, \dots, a_q$ are elements of $\field$ and the morphism $f_i$ is a non zero composition of $u_{k}^{R}$ morphisms. In other words, there exists a sequence $k^{i}_1, \dots k^{i}_{p_{i}}$ such that $f_i = u_{k^{i}_{p_{i}}}\circ\dots\circ u_{k^{i}_{1}}$. Suppose $F(\sum_{i=1}^q a_i\cdot f_i) = 0$ in $\mathcal{Y}_{m,n}$. We can also assume $f_1\not = 0$ and $\alpha_1 \not = 0$. Lemma \ref{lem:relZeroSolid} ensures that zero relations, \emph{i.e.} when $q=1$ correspond exactly to those described in Lemma \ref{lem:RelIdentification}. To conclude when $q>1$ we want to argue that $k^{i}_1, \dots k^{i}_{p_{i}}$ contain the same elements, independently of $i$.
\par By Lemmas \ref{prop:factHomYildCat}, \ref{lem:decompExt} and \ref{lem:decomHom0} there exists a list $K$ of transformations to go from $\alpha$ to $\beta$. We argue that it is the only possible list. Suppose $L$ is a list of transformations giving a non zero morphism from $\alpha$ to $\beta$. Order both $K$ and $L$ using the order relation $<_{min}$ with starting point $k_{min}$ the minimum of the starting points for $K$ and $L$. If there exists $s$ an element of $L$ which is not in $K$ pick $s$ minimal for $<_{min}$. Then using the transformations listed by $L$, $s-1 \in \beta$ but through $C$, $s-1\not\in \beta$ which is a contradiction. Hence all the elements of $L$ appear in $K$. Symmetrically, all the elements of $K$ appear in $L$.
\par Denote $f^{\star}$ the composition of the maps associated to the (canonical) list in increasing order of their combinatorial datum. Note that Lemma \ref{lem:ordering} and its proof apply to the category $C$ as well since it only uses equation (\ref{eq:rel}). Hence every morphism in $C$ from $(R, l)$ to $(R', l')$ is proportional to $f^{\star}$. In particular, there exist elements $a, b\in \field$ such that $f_1 - a\cdot f^{\star}=0$ in $\mathcal{C}$ and $-\frac{1}{a_1}\sum_{i=2}a_i\cdot f_i - b\cdot f^{\star}$ in $\mathcal{C}$. Because $F(\sum_{i=1}^q a_i\cdot f_i) = 0$ in $\mathcal{Y}_{m,n}$, $a= -b$ and $\sum_{i=1}^q a_i\cdot f_i = 0$ in $C$.
\end{proof}
\begin{nota}\label{nota:signs}
For a configuration $R$ as well as an integer $0 \leq l \in R$ and $0 > k\in R$, define \[\kappa(R, l) = \sum_{\substack{x \in R\\ l \geq x\geq 0}} x \text{ and } \nu(R, k) = \sum_{k \geq x\in R} x.\] We establish a number of identities concerning $\kappa$ and $\nu$ combined with a transformation $\sigma_l^-$. Let $l_1$ and $l_2$ be allowed in $R_+$. Without loss of generality we can assume that $l_1 < l_2$. Then $\kappa(\sigma^-_{l_1}(R), l_2) = \kappa(R, l_2) - 1$ while $\kappa(\sigma^-_{l_2}(R), l_1) = \kappa(R, l_1)$. Hence $\kappa$ is a \emph{combinatorial statistic that detects the event $l_1\leq l_2$}. Similarly, if $k_1$ and  $k_2$ are allowed in $R_-$ and $k_1 < k_2$ then $\nu(\sigma^-_{k_1}(R), k_2) = \nu(R, k_2) - 1$ while $\nu(\sigma^-_{k_2}(R), k_1) = \nu(R, k_1)$. Hence $\nu$ detects the event $k_1 < k_2$. Combining the two statistics we get $\kappa(\sigma^-_k(R), l) = \kappa(R, l)$ and $\nu(\sigma^-_l(R), k) = \nu(R, k)$.
Alternatively we can define these quantities on partitions. We will only use the quantity $\kappa$ which we express as follows: let $\alpha$ be a partition, $i$ be allowed in $S_{\alpha}$ and set
\begin{equation}\label{eq:shift}
\kappa(\alpha, i) = \sum_{k = 1}^i \lambda_k.
\end{equation}
Because it will be used several times, we write $\kappa_{\alpha} = \kappa_R = \kappa(\alpha, r)$.
\end{nota}

To conclude this subsection, we give three presentations of the category $\mathcal{Y}_{m,n}$ by generators and relations. The first one is a direct corollary of Proposition \ref{prop:RelYildCat} with the relations of Lemma \ref{lem:RelIdentification}. According to equation (\ref{eq:rel}), with that presentation some squares of irreducible morphisms commute and some anticommute. Using Notation \ref{nota:signs} we give two more presentations, one where all the squares commute and one where they all anticommute. 
\begin{cor}\label{cor:Pres}
The morphisms in $\mathcal{Y}_{m,n}$ are generated by
\begin{enumerate}
	\item\label{rel1}the maps $u_k^R$ for all $R\in C_{m,n}$ and for all $k\in R$ allowed, with relations generated by $\rho_{k,l}^R = u_l^{\sigma_k^-(R)}\circ u_{k}^R - \varepsilon u^{\sigma^-_l(R)}_k \circ u_l^R $ where $k, l\in R$ such that $k - 1, l - 1 \not\in R$ and $\varepsilon = -1$ if $k$ and $l$ are positive and $1$ otherwise, along with $z_{k}^R = u_{k}^{\sigma^-_{k-1}(R)}\circ u_{k}^R$ where $k\in R$, but $k-1, k-2 \not\in R$;
	\item\label{rel2}the maps $v_{k}^R = (-1)^{\kappa(R, k)}\cdot u_k^R$ for all $R\in C_{m,n}$ and for all $k\in R$ allowed, with relations generated by $(\rho'_{k,l})^{R} = v_l^{\sigma_k^-(R)}\circ v_{k}^R - v^{\sigma^-_l(R)}_k \circ v_l^R $ where $k, l\in R$ such that $k - 1, l - 1 \not\in R$ along with $(z'_{k})^R = v_{k-1}^{\sigma^-_k(R)}\circ v_{k}^R$ where $k\in R$, but $k-1, k-2 \not\in R$;
	\item\label{rel3}the morphisms $w_{k}^R = \varepsilon(R,k)(-1)$ for all $R\in C_{m,n}$ with $\varepsilon(R, k) = (-1)^{\nu(R, k) +\kappa_R}$ if $k < 0$ and $\varepsilon(R, k) = 1$ otherwise, for all $k\in R$ allowed, with relations generated by $(\rho_{k,l}'')^R = w_l^{\sigma_k^-(R)}\circ w_{k}^R + w^{\sigma^-_l(R)}_k \circ w_l^R $ where $k, l\in R$ such that $k - 1, l - 1 \not\in R$ along with $(z_k'')^R = w_{k-1}^{\sigma^-_k(R)}\circ  w_{k}^R$ where $k\in R$, but $k-1, k-2 \not\in R$.
\end{enumerate}
\end{cor}
\begin{proof}
The presentation with relations as per item \ref{rel1} follows directly from the previous discussion. The morphisms $\{v_k^{R}\}_{k,r}$ and $\{w_k^{R}\}_{k,r}$ still form generating sets for the morphisms differing from their $u_{k}^R$ counterpart by a unit. Moreover, the zero relations $z^R_k, (z'_k)^R$ and $(z_k'')^R$ are also proportional for each $R$ and $k$. For a computation completing the proof we refer to \cite[Corollary 5.2.14]{gottesman2024these}. However this also follows from a more general yet less explicit argument from \cite[Proof of Prop. 2.2.5]{dyckerhoff_jasso_lekili_2021}.
\end{proof}
%%%%%%%%%%%%%%%%%%%%%%%%%%%%%%%%%%%%%%%%%%%%%%%%%%%%%%%%%%%%%%%
\subsection{Tilting object}
Using Notation \ref{nota:signs} we can finally set
\begin{equation}\label{eq:tiltings}
T: = \bigoplus_{\alpha\in J_{m,n}} \PRes_{\alpha}[\kappa_{\alpha}].
\end{equation}
Note that the morphisms and extensions between couples of indecomposable summands of $T$ only appear in one degree. By shifting each summand of $T$ by $\kappa_{\alpha}$ we concentrate all the morphisms in degree zero. %There are subgraphs of Figure(\ref{fig:hom-graph_2_2}) for which such a procedure is not possible. 
	\begin{lem}\label{lem:shiftTilting}The object $T$ has no self extensions.
	\end{lem}
	\begin{proof}
		First note that the quantity $\kappa_{\alpha}$ only depends on the non zero values of the partition $\alpha$ \emph{i.e.} those indexed by $S_{\alpha}$. Hence if there is a morphism in degree zero $\PRes_{\alpha}\to \PRes_{\beta}$ then $\kappa_{\alpha} = \kappa_{\beta}$. This follows from Lemma \ref{lem:hom0}(\ref{item:4}) and the fact that we only look at plain partitions. %and the fact that the quantity $\kappa_{\alpha}$ does not depend on the multiplicity of zero. 
Moreover, if $J$ is allowed, $|J| = p$ then $\kappa_{q_J(\alpha)} = \kappa_{\alpha} + p$. We put these two remarks together. If there exists a non zero morphism 
			\begin{equation}
			\PRes_{\alpha}[\kappa_{\alpha}] \to \PRes_{\beta}[\kappa_{\beta}][p]
			\end{equation}
		by Proposition \ref{prop:factHomYildCat}, there exists a unique subset $J$ of $S_{\alpha}$ $q_J(\alpha) \in [f(\beta), \beta]$ and $|J| = -\kappa_{\alpha} + \kappa_{\beta} + p$. The subset $J$ is allowed for $\alpha$, so $\kappa_{\alpha} = \kappa_{q_J(\alpha)} - |J|$. Moreover there is a non zero morphism of modules from $\PRes_{q_{J}(\alpha)}$ to $\PRes_{\beta}$ meaning we have $\kappa_{\beta} = \kappa_{q_J(\alpha)}$. Thus, $|J| = \kappa_{\beta} - \kappa_{\alpha}$, $p=0$ and $T$ has no self extensions.
\end{proof}

Proposition \ref{prop:thicc} together with Lemma \ref{lem:shiftTilting} show that the object $T$ is tilting. We now describe its algebra of endomorphisms and to do so we recall the construction of higher Auslander algebras of type $A$ following convention from \cite[Definition~2.12]{herschend2021auslander}. Note that we compose arrows using a different convention but everything else is written as close to that source as possible. Let $d$ and $s$ be integers. The \defn{higher Auslander algebra of type $A_s^d$} is constructed as a quotient of a quiver algebra by relations. The underlying set of the quiver is the set of increasing sequences of length $d+1$ with values in $\{1, \dots, d+s\}$. Let $x = (x_0,\dots, x_d)$ be an element of $Q_0$. For a value $k$ appearing in a sequence $x$, we define a partial transformation $\sigma_k^+$ on $Q_0$ by
\[\sigma_k^+(x) = (x_0<\dots x_{i} < k + 1 < x_{i+2} < \dots x_d)\]
whenever the resulting sequence is increasing \emph{i.e.} $x_{i+2} \not = k + 1$. Similarly we write $\sigma^-_k$ for the partial  map that replaces $k$ by $k-1$ in the sequence whenever possible. Let $Q_1$ the set of arrows of the quiver consist of elements $\alpha_k^x$ with source $x$ and target $\sigma_k^+(x)$ whenever the target is well defined. In the path algebra of the resulting quiver $\IncAlg(Q)$ we consider the ideal generated by the following elements
\begin{equation*}
\rho^x_{k,l} =  \begin{cases}
\alpha_k^{\sigma_l^+(x)}\alpha_l^x - \alpha_l^{\sigma_k^+(x)}\alpha_k^x & \text{ if } k, l\in x \text{ and } k + 1, l+1 \not \in x,\\
\alpha_k^{\sigma_{k+1}^+(x)}\alpha_{k+1}^x&\text{ if } l= k+1 \in x \text{ and } l + 1\not\in x.\\
\end{cases}
\end{equation*}
We denote by $G$ the vector space generated by these elements. Then we set $(\IncAlg(Q)/ I)^{op}$ to be the higher Auslander algebra $A_s^d$. With these relations and the usual grading by length of paths it is clear that $A_s^d$ is quadratic. Its quadratic dual is 
\[A^{!} = (\IncAlg(Q)^{op}/\langle G^{\bot}\rangle^{op}\]
where $\IncAlg(Q)^{op}= \IncAlg(Q^{op})$ and $G^{\bot}$ is the orthogonal complement of $G$ in the dual of $\IncAlg(Q)_{2}$, the vector space with basis the paths of length two in the quiver $Q$ \cite{martinezvilla2004dual}. It remains to compute the orthogonal of $G$ in $\IncAlg(Q)^{op}$ to get a presentation of the quadratic dual as a quiver with relations. 

\begin{prop}\label{prop:OrthRels} The orthogonal component $G^{\bot}$ of $G$, has basis
	\begin{equation*}
		\begin{cases}
		(\alpha_{k}^x)^{op}(\alpha_{k+1}^{\sigma_k^+(x)})^{op}\\
		(\alpha_k^x)^{op}(\alpha_l^{\sigma_k^+(x)})^{op} + (\alpha_l^x)^{op}(\alpha_k^{\sigma_l^+(x)})^{op}\\
		\end{cases}
	\end{equation*}
where $k, l\in x$ while $k+1, l+1\not\in x$. 
\end{prop}

\begin{proof}
It is clear that these elements are in $G^{\bot}$. We know that $$\dim G + \dim G^{\bot} = \dim \IncAlg(Q)_2.$$ We argue that the set above is precisely $G^{\bot}$ for reasons of cardinality. To do so, notice that the composition of two arrows in the quiver modifies either one or two elements of the sequence. The case where it modifies one element corresponds to the first relations in the equation above. When two elements are modified there are two cases. Either the order matters or it does not. If it does then we are in the case of the zero relations of $A_s^d$. If it doesn't, then the 2-path appears in the commutation relation of $A_s^d$ but also in the anticommutation relation we exhibited for $(A_s^d)^{!}$ just now. Hence there is a partition of a basis of $\IncAlg(Q)_2$ into the paths in $G$ and $G^{\bot}$.
\end{proof}
Note that in the quadratic dual, squares commute with a sign \emph{i.e.} $ab + cd = 0$. What remains to prove for Theorem \ref{thm:DerEqTypeA} is that the signs of the squares can be modulated meaning that we can construct isomorphisms between the quiver algebra modulo relations with $ab + cd = 0$ and then one with $ab - cd = 0$ for some square relations in the ideal $I$. This is not true for general quotients of paths algebras but it holds for certain configurations of squares in $J_{m,n}$  and $A_s^d$. Note that $A_s^d$ and $(A_s^d)^{op}$ are isomorphic.

\begin{proof}[Proof of Theorem  \ref{thm:DerEqTypeA}]
Consider the presentation of $\mathcal{Y}_{m,n}$ given in Corollary \ref{cor:Pres}(\ref{rel2}). We restrict it to the indecomposable components of the object $T$. Like in \cite{dyckerhoff_jasso_lekili_2021}, we can then replace configurations by their complements  in $\mathcal{Z}$. Next we shift them by $+m$. This way we recognise the presentation of $A_{m+1}^{n-1}$. See Figure \ref{fig:complement} to see how taking the complement of the configuration gives a relation that matches one in $A_{m+1}^{n-1}$. The boxes in Figure \ref{fig:complement} indicated empty spots in the abacus. When taking complements, the boxes become the elements of the resulting configuration and the dots become the empty spots of the corresponding abacus.
\end{proof}
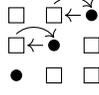
\begin{figure}
\centering
\begin{tikzpicture}
%first line
\filldraw (.5, 0.6)node[](circle1){} circle (2pt);
\draw (-0.1, 0.5) rectangle node[](rect1){} (0.1,0.7);
\draw (-0.6, 0.5) rectangle (-0.4,0.7);
\draw[->](circle1) -- (rect1); 
\path[->](0,0.75) edge [bend left = 50] (.5,.7);
%second line
\draw (0.4, 0.1) rectangle (0.6, 0.3);
\filldraw (0, 0.2) node[](circle2){} circle (2pt);
\draw (-0.6, 0.1) rectanglenode[](rect2){} (-0.4,0.3);
\draw[->](circle2) -- (rect2); 
\path[->](-0.5,0.35) edge [bend left = 50] (0,.3);
%third line
\draw (0.4, -0.3) rectangle (0.6, -0.1);
\draw (-0.1, -0.3) rectangle (0.1,-0.1);
\filldraw (-0.5, -0.2) circle (2pt);
\end{tikzpicture}
\caption{illustrating the duality of the relations}\label{fig:complement}
 \end{figure}

 \begin{prop}The algebra $J_{m, n}$ is derived equivalent to $(A_{n+1}^{m-1})^{!}$.
\end{prop}
\begin{proof}
Consider the generators and relations of the algebra $\End(T)^{op}$ given in Corollary \ref{cor:Pres} item \ref{rel3} and shift the values of the configurations by $+m$. Then we get the generators and relations of the quadratic dual of $A_{n+1}^{m-1}$ described in Proposition \ref{prop:OrthRels} after\end{proof}

We also want to make explicit the following result which follows also from an argument in the proof of \cite[Proposition 2.2.5]{dyckerhoff_jasso_lekili_2021}.
\begin{cor}There is an isomorphism of algebras between $A_s^d$ and $(A_{d+2}^{s-2})^{!}$.
\end{cor}

%\begin{note}
%Also the original signs of our OG presentation are a mixture of the signs of the higher auslander algebra and the sign of the correct quadratic dual. We can also state here the isomorphism between the two, because it is indeed not written in \cite{dyckerhoff_jasso_lekili_2021}, but as a concequence of their work and not this one. note that the proof in the paper by Dyckerhoff jasso and Leikili relies on a different description fo the higher auslander algebra (not with quiver and rels, but with a vector space basis).
% \end{note}
%%%%%%%%%%%%%%%%%%%%%%%%%%%%%%%%%%%%%%%%%%%%%%%%%%%%%%%%%%%%%%

%\appendix
%\input{sections/proofs}
%
%\Urlmuskip=0mu plus 3mu\relax
\printbibliography

{\footnotesize \textsc{Tal Gottesman\\
Universit\'e Paris Cit\'e and Sorbonne Universit\'e, CNRS, IMJ-PRG, F-75013 Paris, France.}}\\
\emph{E-mail address:} \texttt{tal.gottesman@rub.de}
\end{document}